\documentclass[12pt]{article}
\usepackage[utf8]{inputenc}
\usepackage[english]{babel}
\usepackage{amsthm}
\usepackage{amsmath, amssymb, mathrsfs}
\usepackage{hyperref}
\usepackage{cleveref}
\usepackage{xypic}
\usepackage{graphicx}
\usepackage{amsfonts}
\usepackage{aliascnt}
\newtheorem{theorem}{Theorem}[section]
\newtheorem{corollary}[theorem]{Corollary}
\newtheorem{lemma}[theorem]{Lemma}
\theoremstyle{definition}
\newtheorem{definition}[theorem]{Definition}

\theoremstyle{definition}
\newtheorem*{definition*}{Notational Conventions}

\theoremstyle{remark}
\newtheorem{remark}[theorem]{Remark}

\theoremstyle{definition}
\newtheorem{example}[theorem]{Example}

\theoremstyle{definition}

\theoremstyle{definition}

\title{Descent Theory and Mapping Spaces}
\author{Nicholas J. Meadows}  
\date{\today}  
\begin{document}

\maketitle

\section{Introduction}
\label{intro}

A $2$-category is a generalization of an ordinary category in which there are morphisms between morphisms. More generally, a strict $\infty$-category is an object in which there are notions of $n$-morphisms between $(n-1)$-morphisms for all integers $n$. In applications, we often consider (weak) $(\infty, n)$-categories, in which associativity and identity hold only up to isomorphism at the next level, and the morphisms below level $(n+1)$ are invertible. This is because $\infty$-categories can be studied using the techniques of homotopy theory. For instance, Kan complexes (i.e. spaces) are models for $\infty$-groupoids (higher categories where all morphisms are invertible up to homotopy). 

Traditionally, a stack is defined to be a sheaf of groupoids on a small Grothendieck site $\mathscr{C}$ in which we can glue objects together along compatible families of isomorphisms. More precisely, a stack is a sheaf of groupoids that satisfies the effective descent condition (see \cite{Giraud}). A higher stack is a presheaf of $(\infty, n)$-categories that satisfies glueing conditions involving higher morphisms.

A systematic theory of higher stacks with values in $\infty$-groupoids was developed by Jardine, which was useful in algebraic K-theory and cohomology theory (see \cite{Simpson-Descent}). Lurie developed the theory of $(\infty, 0)$-stacks (i.e. stacks valued in $\infty$-groupoids) in \cite{Lurie} internally to a quasi-category. Lurie's approach has the advantage of being more conceptual and general, whereas Jardine's approach is closer to the classical geometric language. 

Generalizing Jardine's work in a different direction, Hirschowitz and Simpson developed a theory of $(\infty, n)$-stacks (stacks with values in $(\infty, n)$-categories) for all $n \in \mathbb{N}$ in \cite{Simpson}.  They used the iterated Segal construction to construct the homotopy theory of $(\infty, n)$-categories.
  The main example of higher stacks that they produce are $(\infty, 1)$-stacks of the form $$L(M),$$ where $M$ is a presheaf of model categories,  and $L$ is a simplicial localization functor (i.e. a functor that replaces a category with weak equivalences with a higher category), applied sectionwise. Hirschowitz and Simpson give a sufficient condition for $LM$ to be a higher stack (\cite[Theorem 19.4]{Simpson}). Intuitively, this can be interpreted as a statement about how objects glue together along weak equivalences, since the object $LM$ acts as a higher categorical approximation to $M$.

The purpose of this paper is to give an account of the theory of $(\infty, 1)$-stacks in the sense of \cite{Simpson}, using quasi-categories and the local Joyal model structure of \cite{Nick}. The former is advantageous since many quasi-categorical constructions are more tractable than their Segal category analogues. The latter allows us to simplify the sheafification part of the theory by exploiting the technique of Boolean localization, which was used extensively in \cite{Nick} and \cite{Nick2}. This is one of the key contributions of this paper. 

The main idea of the paper is to characterize descent for the local model structure in terms of mapping space presheaves and an effective descent condition. For sites satisfying certain mild hypotheses, it suffices to show effective descent for covers generated by a single morphism $\phi : V \rightarrow U$. Effective descent for a presheaf of quasi-categories $X$ along such a cover reduces to showing the essential surjectivity of
 $$X(U) \rightarrow \mathrm{holim}_{n \in \Delta^{*} } X(\tilde{C}(\phi)_{n}),$$
where the homotopy limit is taken along the Cech nerve $\tilde{C}(\phi)$. These facts are combined to yield Theorem~\ref{thm4.12}, the main result of this paper. 

An important application of this theorem is to descent for left Quillen presheaves (Theorem~\ref{thm5.8}). Essentially, one replaces the condition on $\mathrm{holim}$ in Theorem~\ref{thm4.12} with a condition of essential surjectivity from the downstairs value, to the homotopy category of lax sections of the Quillen presheaf over the covering. This is an important statement when one wants to calculate anything. It is applied in the final section of the paper to prove a descent property for unbounded complexes. This is a strengthening of \cite[Section 20]{Simpson}, which showed descent for non-negatively graded complexes.  

This paper is organized as follows. In Section~\ref{sec3}, we explain the background on the Joyal model structure, and various quasi-categorical results needed for the rest of this paper.  The most important results are a description of the homotopy coherent nerve functor $\mathfrak{B}$ and its homotopy-theoretic significance. 

Section~\ref{sec4} reviews both the Jardine model structure on simplicial presheaves and the local Joyal model structure of \cite{Nick}. These are used to define a higher stack condition (quasi-injective descent), which is then related to the usual notion of hyperdescent (see Remark~\ref{rmk2.17} for a discussion).

The fifth section of this paper is devoted to describing the effective descent condition for a presheaf of quasi-categories (Definition~\ref{def3.1}). A sequence of lemmas relates effective descent for a covering with a $\mathrm{holim}$ taken over the covering. In the case of a cover $R$ generated by a single element, the finality of an inclusion $\Delta \subseteq R$ allows us to conclude the precise form of the condition used in Theorem~\ref{thm4.12}.

In the sixth section we prove the main results on descent for presheaves of quasi-categories (Theorem~\ref{thm4.10}, Theorem~\ref{thm4.12}). These are analogues of \cite[Theorems 10.2 and 19.4]{Simpson}. An interesting intermediate result is Theorem~\ref{thm4.9}, which says that we can check effective descent on the underlying presheaf of Kan complexes $J(X)$. Intuitively, this is because descent data involve specifying
equivalences on the higher intersections of open sets in the covering.
 
In Section~\ref{sec7}, we treat the case of presheaves of model categories. A result of Bergner allows one to compare the $\mathrm{holim}$ with the homotopy category of lax sections (see Theorem~\ref{thm5.5}). This leads to the application to left Quillen presheaves in Theorem~\ref{thm5.8}.

In the final section, we use our results to show that unbounded complexes on a reasonable Grothendieck site form a higher stack (Corollary~\ref{cor6.12}; see also Theorem~\ref{thm6.11}). 

\section{Notational and Terminological Conventions}\label{notation}
Given a category $C$, we write $BC$ for its nerve. We write $\mathrm{Iso}(C)$ for the subcategory of $C$ whose morphisms are isomorphisms of $C$. Given a category $C$ and $b \in \mathrm{Ob}(C)$, we write $C/b$ for the usual \textbf{slice category} over $b$. Let $\mathrm{sSet}, \mathrm{sCat} $ denote the categories of simplicial sets and simplicial categories, respectively. Let $\mathrm{Cat}$ denote the category of small categories.

Given two simplicial sets $X, Y$, write $X^{Y}$ for the internal hom in simplicial sets.
We call a map of simplicial sets which has the right lifting property with respect to the horns $\Lambda_{i}^{n} \subseteq \Delta^{n}, 0 < i \le n$ a \textbf{right fibration}. A map which has the left lifting property with respect to right fibrations is called a \textbf{right anodyne map}. A map which has the right lifting property with respect to the horns $\Lambda_{i}^{n} \subseteq \Delta^{n}, 0 < i < n$ is called an \textbf{inner fibration}. A map which has the left lifting property with respect to inner horns is called \textbf{inner anodyne.}

The \textbf{path category functor} $P : \mathrm{sSet} \rightarrow \mathrm{Cat}$ is the left adjoint to the nerve functor.
We write $\pi(X)$ for the \textbf{fundamental groupoid} of a quasi-category $X$, which is defined to be the groupoid completion of $P(X)$.

We call a simplicially enriched category a \textbf{simplicial category}. Given a simplicial category $\textbf{M}$, and objects $x, y$ of $\textbf{M}$, we write $\textbf{hom}_{\textbf{M}}(x, y)$ for the simplicial set of morphisms between $x$ and $y$. Given a simplicial category $\textbf{M}$, we write $\pi_{0}(\textbf{M})$ for its \textbf{fundamental category}. That is, the category whose objects are objects of $\textbf{M}$, such that for each $x, y \in C$, $\textbf{hom}_{\pi_{0}(\textbf{M})}(x, y)  = \pi_{0}\textbf{hom}_{\textbf{M}}(x, y)$. A \textbf{fibrant simplicial category} is a simplicial category which is fibrant in the Bergner model structure; i.e., all of its simplicial homs are Kan complexes. The Bergner model structure and the constructions for simplicial categories above can be found in \cite{BERGNER1}. 

A simplicial category can be viewed as a simplicial object in the category $\mathrm{Cat}$ of small categories, discrete on objects. Given a simplicial category $\textbf{M}$, we will write $\textbf{M}_{n}$ for $\textbf{M}$ in simplicial degree $n$.

\section{Preliminaries on the Joyal model structure}\label{sec3}
The purpose of this section is to review the Joyal model structure on simplicial sets. One particularly important feature is the characterization of weak equivalences in Theorem~\ref{thm1.6}. Another major point is the Quillen equivalence between the Joyal and Bergner model structures. The right adjoint, the homotopy coherent nerve, gives an important example of a simplicial localization functor. We will use this in Section~\ref{sec7}. 
\\

The \textbf{Joyal model structure} on simplicial sets, whose existence is
 asserted in \cite[Theorem 2.2.5.1]{Lurie} and
 \cite[Theorem 6.12]{Joyal-quasi-cat}, is one of the main models for the homotopy theory of $(\infty, 1)$-categories. The fibrant objects of this
 model structure are the \textbf{quasi-categories}: simplicial sets $X$ such that the map $X \rightarrow *$ is an inner fibration. 
The cofibrations are the monomorphisms. 
The weak equivalences are called \textbf{Joyal equivalences} and can be described as follows. Given two simplicial sets $K$ and $X$, $\tau_{0}(K, X)$ 
will denote \textbf{Joyal's set}, which is defined to be the set of isomorphism classes of objects in $P(X^{K})$. 
The Joyal equivalences are defined to be maps
 $f: A \rightarrow B$ such that, for each quasi-category $X$, the map 
\begin{equation*}
\tau_{0}(B, X) \rightarrow \tau_{0}(A, X)
\end{equation*}
is a bijection. The fibrations of this model structure are called
 \textbf{quasi-fibrations}. The trivial fibrations are the trivial Kan
 fibrations.

 For a quasi-category $X$, let $J(X)$ denotes its maximal Kan subcomplex. Write $I = B\pi( \Delta^{1})$. 
The following is taken from \cite[Section 1]{Joyal1}:

\begin{theorem}\label{thm1.0}
A quasi-category $X$ is a Kan complex iff $P(X)$ is a groupoid. Thus, $J(X)$ can be constructed by taking the maximal subcomplex of $X$ whose 1-simplices are invertible in $P(X)$. Furthermore, a 1-simplex $s: \Delta^{1} \rightarrow X$ is invertible in $P(X)$ iff it extends to a map $I \rightarrow X$. 
\end{theorem}

\begin{definition}\label{def1.1}
Given two simplicial sets $S$ and $T$, their \textbf{join}, denoted $S* T$, is a simplicial set whose $n$-simplices are described by the formula
$$
(S * T)_{n} = S_{n} \cup T_{n} \bigcup_{i + j = n-1} (S_{i} \times T_{j}).
$$ 
The ith degeneracy $d_{i} : (S * T)_{n} \rightarrow (S * T)_{n-1}$ is defined on the factors $S_{n}$ and $T_{n}$ using the degeneracy maps of $S_{n}$ and $T_{n}$. 
For $(\sigma, \sigma') \in S_{i} \times T_{j}$, we have the formula 

\[ d_{k}(\sigma, \sigma') = \left\{
  \begin{array}{lcr}
(d_{k}\sigma, \sigma') & \text{if} & k \le i, i \neq 0 \\
 (\sigma, d_{k-i-1}(\sigma')) & \text{if} & k > i, j \neq 0 
\end{array}
\right.
\]
\end{definition}

\begin{definition}\label{def1.2}
Given a map of simplicial sets $p : K \rightarrow S$, there is a simplicial set $S_{/p}$ such that 

$$(S_{/p})_{n} = \mathrm{hom}_{p}(\Delta^{n} * K, S),$$
where $\mathrm{hom}_{p}$ means simplicial set maps $\phi$ such that $\phi|_{K} = p$. We call this the \textbf{slice} over $p$.  
\end{definition}
There is a natural map $S_{/p} \rightarrow S$ which is induced in simplicial degree $n$ by $\Delta^{n} \subseteq \Delta^{n} * K$. We call this the \textbf{projection map}. 

\begin{definition}\label{def1.3}
Given a simplicial set $X$ and $x, y \in X$, the \textbf{mapping space} between $x$ and $y$ is defined to be the pullback 
$$
\xymatrix
{
\mathrm{Map}_{X}(x, y) \ar[r] \ar[d] & \ar[d]^{q} X_{/y}\\
\, * \ar[r]_{x} \ar[r] & X
}
$$
where $q$ is the projection map. 
\end{definition}

In \cite{Lurie}, Lurie writes $\mathrm{Hom}_{X}^{R}(x, y)$ for these mapping spaces and calls them right mapping spaces. 
The mapping spaces in a quasi-category are Kan complexes (\cite[Proposition 1.2.2.3]{Lurie}).

\begin{definition}\label{def1.4}
Let $f: X \rightarrow Y$ be a map of quasi-categories. We say that $f$ is fully faithful iff 
for each $x, y \in X$, $\mathrm{Map}_{X}(x, y) \rightarrow \mathrm{Map}_{Y}(f(x), f(y))$ is a weak equivalence of Kan complexes. We say that $f$ is \textbf{essentially surjective} iff $P(f)$ is essentially surjective.  
\end{definition}

\begin{lemma}\label{lem1.5}
A morphism $f: X \rightarrow Y$ of quasi-categories is essentially surjective iff 
$\pi_{0}J(f)$ is surjective. 
\end{lemma}
\begin{proof}
By Theorem~\ref{thm1.0}, 
a 1-simplex $s: \Delta^{1} \rightarrow X$ is invertible iff it extends to a map $B\pi \Delta^{1} \rightarrow X$. Thus, there are bijections (natural in $X$)
$$
\pi_{0}\mathrm{Iso}(P(X)) \cong \pi_{I}(*, X) \cong \pi_{\Delta^{1}}(*, J(X)) \cong \pi_{0}J(X),
$$
where $\pi_{I}$ denotes the homotopy classes of maps with respect to $I = B\pi \Delta^{1}$ and $\pi_{\Delta^{1}}$ those with respect to $\Delta^{1}$.  
\end{proof}

The mapping space construction is important because of the following result:

\begin{theorem}\label{thm1.6}
\textnormal{(}see \textnormal{\cite[Theorem 2.2.5.1]{Lurie}, \cite[Theorem 8.1]{MappingSpaces})}.
Suppose that $f: X \rightarrow Y$ is a map of quasi-categories. Then $f$ is a Joyal equivalence iff it is fully 
faithful and essentially surjective. 
\end{theorem}

\begin{definition}\label{def1.7}
Suppose that $\Omega^{*}$ is a cosimplicial object in a category $C$. Then there is a pair of adjoint functors associated to $\Omega^{*}$
$$
 |\, \, |_{\Omega^{*}} : \mathrm{sSet} \rightleftarrows C : \mathrm{Sing}_{\Omega^{*}}.   
$$
The left adjoint is given by  
$$
|S|_{\Omega^{*}} = \underset{\Delta^{n} \rightarrow S}{\underset{\longrightarrow}{\mathrm{lim}}} \Omega^{n}
$$
and the right adjoint is given by $\mathrm{Sing}_{\Omega^{*}}(S)_{n} = \mathrm{hom}(\Omega^{n}, S)$. The right adjoint is known as the \textbf{singular functor associated to $\Omega^{*}$}.
\end{definition}

For each $n \in \mathbb{N}$ there is a simplicial category $\Phi^{n}$ such that:
\begin{enumerate}
\item{
The objects $\Phi^{n}$ are the objects in the set $\{ 0, 1, \ldots, n\}$.}
\item{ $\mathrm{hom}_{\Phi^{n}}(i, j)$ can be identified with the nerve of the poset $\mathcal{P}_{n}[i, j]$ of subsets of the interval $[i, j]$ which contains the endpoints. That is, $\mathrm{hom}_{\Phi^{n}}(i, j) \cong (\Delta^{1})^{i-j-1}$.} 
\item{Composition is induced by union of posets.}
\end{enumerate}

These $\Phi^{n}$ glue together to give a cosimplicial object $\Phi^{*}$. The singular functor associated to $\Phi^{*}$ is called the \textbf{homotopy coherent nerve}, and is denoted $\mathfrak{B}$. We write $\mathfrak{C}$ for its left adjoint.

  We call a map $f: \mathcal{C} \rightarrow \mathcal{D}$ of simplicial categories a \textbf{DK-equivalence} iff 
\begin{enumerate} \item{$\mathrm{hom}_{\mathcal{C}}(x, y) \rightarrow \mathrm{hom}_{\mathcal{D}}(f(x), f(y))$ are weak equivalences $\forall \, x, y \in \mathrm{Ob} (\mathcal{C})$.} 
\item{ $\pi_{0}(\mathcal{C}) \rightarrow \pi_{0}(\mathcal{D})$ is an equivalence of categories.}
\end{enumerate}

In \cite{BERGNER1}, Bergner produced a model structure on simplicial categories in which the weak equivalences are the DK-equivalences. The fibrant objects are the objects with fibrant simplicial hom's. 
The homotopy coherent nerve is significant because of the following theorem, which relates the Bergner model structure to the Joyal model structure:

\begin{theorem}\label{thm1.8}
There is a Quillen equivalence
$$
\mathfrak{C} : \mathrm{sSet} \rightleftarrows \mathrm{sCat} : \mathfrak{B}
$$
between the Joyal model structure and the Bergner model structure. 
\end{theorem}

\begin{theorem}\label{thm1.9}

Suppose that $\textnormal{\textbf{M}}$ is a fibrant simplicial category. Then: 
\begin{enumerate}
\item{There is a natural isomorphism $\pi_{0}\mathfrak{C} \cong P$. In particularly, we have equivalences of categories $ \pi_{0}\textnormal{\textbf{M}} \simeq \pi_{0}\mathfrak{C}\mathfrak{B}(\textnormal{\textbf{M}}) \cong P\mathfrak{B}(\textnormal{\textbf{M}})$.}
\item{The Kan complexes
$
\mathrm{Map}_{\mathfrak{B}(\textbf{M})}(x, y)$ and $\textnormal{\textbf{hom}}_{\textbf{M}}(x, y) $ are connected by a natural zig-zag of weak equivalences. 
}
\end{enumerate}
\end{theorem}

\begin{proof}

A simplicial set is a colimit of its non-degenerate simplices and we have isomorphisms (natural in ordinal numbers $n$)
$$
P(\Delta^{n}) \cong [\textbf{n}] \cong \pi_{0}\mathfrak{C}(\Delta^{n}).
$$
$\pi_{0}\mathfrak{C}$ and  $P$ are both left adjoints and thus preserve colimits. Thus, we have
$$
P \cong \pi_{0} \mathfrak{C}.
$$
The remainder of 1. now follows from Theorem~\ref{thm1.8}. 

By \cite[Proposition 2.2.2.7, 2.2.2.13]{Lurie} there exists a cosimplicial object $Q^{*}$ and a natural weak equivalence
$$
|\mathrm{Sing}_{Q^{*}}(\textbf{hom}_{\textbf{M}}(x, y))|_{Q^{*}} \cong |\mathrm{Map}_{\mathfrak{B}(\textbf{M})}(x, y)|_{Q^{*}} \rightarrow     \mathrm{Map}_{\mathfrak{B}(\textbf{M})}(x, y).
$$

On the other hand, since mapping spaces in a quasi-category are Kan complexes, we also have a weak equivalence

$$
|\mathrm{Sing}_{Q^{*}}(\textbf{hom}_{\textbf{M}}(x, y))|_{Q^{*}} \rightarrow \textbf{hom}_{\textbf{M}}(x, y)
$$
 by \cite[Corollary 2.2.2.10]{Lurie}.
\end{proof}

\section{Local Homotopy Theory}\label{sec4}

The purpose of this section is to review the definitions of two model structures on simplicial presheaves on a Grothendieck site: the Jardine and local Joyal model structures. The weak equivalences in these model structures are, respectively, called local weak equivalences and local Joyal equivalences. In the case our topos has enough points, a map is a local weak equivalence (resp. local Joyal equivalence) iff it induces weak equivalences (Joyal equivalences) on each stalk. 

We will first review the technique of Boolean localization. This technique is used to define local weak equivalences in full generality. That is, when our topos does not have enough points. Finally, we will review some facts about quasi-injective descent, a higher stack condition, and discuss its relationship to the concept of hyperdescent from \cite{Lurie}.
\\

A \textbf{Grothendieck site} is a category-theoretic generalization of the concept of topological space. It is specified by choosing a category $C$ and a collection of coverings (sets of morphisms $\{ U_{i} \rightarrow U \}_{i \in I}$), subject to various axioms. Grothendieck sites were originally invented to produce new cohomology theories in algebraic geometry, by allowing one to discuss topologies that are finer than the usual Zariski topology on a scheme. A good overview of this concept is found in \cite[Chapter 3]{MM}.   

Throughout the rest of the paper, we fix a small Grothendieck site $\mathscr{C}$. We will write $s\textbf{Pre}(\mathscr{C})$ for the simplicial presheaves on $\mathscr{C}$. We write $s\textbf{Sh}(\mathscr{C})$ for the simplicial sheaves on $\mathscr{C}$. Given a simplicial set $K$ and a simplicial presheaf $X$, write $\mathrm{hom}(K, X)$ for the simplicial presheaf given by $U \mapsto \mathrm{hom}(K, X(U))$. We will identify simplicial sets with constant simplicial presheaves. Given a simplicial set and a simplicial presheaf $K$, we will write $X^{K}$ for the simplicial presheaf $U \mapsto X(U)^{K}$. We denote sheafification by 
$$
L^{2} : s\textbf{Pre}(\mathscr{C}) \rightarrow s\textbf{Sh}(\mathscr{C}).
$$

The technique of \textbf{Boolean localization} is essential to the study of model structures on simplicial presheaves; overviews of this technique are given in \cite[Section 2]{Nick} and \cite[Chapter 3]{local}.
 A Boolean localization  $$p = (p^{*}, p_{*}):  \textbf{Sh}(\mathscr{B}) \rightarrow  \textbf{Sh} (\mathscr{C})$$ 
 is a surjective geometric morphism with $\mathscr{B}$ a complete Boolean algebra equipped with the canonical site. The Boolean localization allows us to reason about local weak equivalences by replacing them with sectionwise weak equivalences, thereby reducing many statements about local equivalences to the classical setting. We will fix a Boolean localization for $\mathscr{C}$, denoted $p$, and write 
 
$$
p^{*} : s\textbf{Sh}(\mathscr{C}) \rightleftarrows : s\textbf{Sh}(\mathscr{B}) : p_{*}
$$
for the adjoint pair obtained by applying the left and right adjoint parts of $p$ levelwise to a simplicial object in sheaves (simplicial sheaf). Note that all Grothendieck topoi have a Boolean localization by a theorem of Barr (\cite[pg. 515]{MM}).

\begin{definition}\label{def2.1}
We call a simplicial set \textbf{finite} iff it has finitely many non-degenerate simplices. 
 Suppose that $i : K \subseteq L$ is an inclusion of finite simplicial sets, and $f: X \rightarrow Y$ is a map of simplicial presheaves. Say that $f$ has the \textbf{local right lifting property} with respect to $i$ if for every commutative diagram

$$
\xymatrix{
 K \ar[r] \ar[d] & X(U)\ar[d]\\
L \ar[r] & Y(U) \\ 
}
$$
there is some covering sieve $R \subseteq \mathrm{hom}(-, U), U \in \mathrm{Ob}(\mathscr{C})$, such that the lift exists in the diagram
$$
\xymatrix{
K \ar[r] \ar[d] & X(U) \ar[r]^{X(\phi)} & X(V) \ar[d]\\
L \ar[r] \ar@{.>}[urr] & Y(U) \ar[r]_{Y(\phi)} & Y(V)\\
}
$$
for every $\phi \in R$. 
\end{definition}

\begin{definition}\label{def2.2}
We say that a map $f :X \rightarrow Y$ of simplicial presheaves is a \textbf{local trivial fibration} iff it has the local right lifting property with respect to the inclusions $\partial \Delta^{n} \rightarrow \Delta^{n}$ for all $n \ge 0$. We call a map which has the local right lifting property with respect to the horn inclusions $\Lambda_{i}^{n} \rightarrow \Delta^{n}$  a \textbf{local Kan fibration}. A map which has the local right lifting property with respect to the inner horn inclusions is called a \textbf{local inner fibration}.
\end{definition}

\begin{remark}\label{rmk2.3}
It is clear that a sectionwise trivial fibration (resp. Kan fibration, inner fibration) is also a local trivial fibration (resp. Kan fibration, inner fibration). 
\end{remark}

\begin{lemma}\label{lem2.4} \textnormal{\cite[Corollaries 4.15, 4.23]{local}}.
A map $f : X \rightarrow Y$ is a local trivial fibration \textnormal{(}resp. local Kan fibration, local inner fibration\textnormal{)} iff $p^{*}L^{2}(f)$ is a sectionwise trivial fibration \textnormal{(}resp. sectionwise Kan fibration, sectionwise inner fibration\textnormal{)}. In particular, if $X$ is a presheaf of quasi-categories \textnormal{(}resp. Kan complexes\textnormal{)}, $p^{*}L^{2}(X)$ is a presheaf of quasi-categories \textnormal{(}resp. Kan complexes\textnormal{)}. 
\end{lemma}

\begin{lemma}\label{lem2.5}
Let $K$ be a finite simplicial set, and $X$ a simplicial presheaf. Then there are natural isomorphisms:
\begin{enumerate}
\item{$p^{*}L^{2}\mathrm{hom}(K, X) \cong \mathrm{hom}(K, p^{*}L^{2}X)$}
\item{$p^{*}L^{2}(X^{K}) \cong p^{*}L^{2}(X)^{K}$.}
\end{enumerate}
\end{lemma}

\begin{proof}
The first statement is clearly true when $K = \Delta^{n}$. It follows in general from the facts that both $p^{*}$ and $L^{2}$ preserve finite limits and a simplicial set is a colimit of its non-degenerate simplices. The second statement is immediate from the first. 
\end{proof}

\begin{definition}\label{def2.6}
We say that a map $f: X \rightarrow Y$ of simplicial presheaves is a sectionwise weak equivalence (resp. sectionwise Joyal equivalence) iff 
$$
X(U) \rightarrow Y(U)
$$
is a weak equivalence (Joyal equivalence) for each $U \in \mathrm{Ob}(\mathscr{C})$.
\end{definition}

Given a simplicial presheaf $X$, we will write $\mathcal{S}_{\mathrm{Joyal}}(X)$ for the presheaf of quasi-categories obtained by applying fibrant replacement for the Joyal model structure sectionwise to $X$. We will also write $\mathrm{Ex}^{\infty}(X)$ for the $\mathrm{Ex}^{\infty}$ functor applied sectionwise to $X$.  

\begin{definition}\label{def2.7}
We say that a map $f: X \rightarrow Y$ is a local weak equivalence (resp. local Joyal equivalence) iff $L^{2}\mathrm{Ex}^{\infty}p^{*}L^{2}(f)$ (resp. $L^{2}\mathcal{S}_{\mathrm{Joyal}}p^{*}L^{2}(f)$) is a sectionwise weak equivalence (resp. sectionwise Joyal equivalence).   
\end{definition}
 
 The intuition behind Boolean localization is that it can be regarded as giving a `fat’ point for a site (for more details see \cite[Section 2]{Boolean}). Thus the definition of local weak equivalence generalizes the idea of stalkwise weak equivalence in the case of a topos with enough points. These definitions of weak equivalence are independent of the choice of Boolean localization.
 
 \begin{theorem}[Jardine]\label{thm2.8}
 There is a model structure on $s\mathbf{Pre}(\mathscr{C})$  in which the weak equivalences are the local weak equivalences and the cofibrations are monomorphisms. This is called the \textbf{Jardine model structure}. The fibrations are known as \textbf{injective fibrations}. 
 \end{theorem}
 
 Given two simplicial presheaves, $X, Y$, write $\textbf{hom}(X, Y)$ for the simplicial set given by 
 $$
\textbf{hom}(X, Y)_{n} = \mathrm{hom}(X \times \Delta^{n}, Y).
$$
 This simplicial hom gives the Jardine model structure the structure of a simplicial model category. 
 
 \begin{theorem}\label{thm2.x} \textnormal{\cite[Theorem 3.3]{Nick}}.
 There is a model structure on $s\mathbf{Pre}(\mathscr{C})$ in which the weak equivalences are the local Joyal equivalences and the cofibrations are monomorphisms. This is called the \textbf{local Joyal model structure.} The fibrations are called \textbf{quasi-injective} fibrations.  
 \end{theorem} 
 
 We collect the following facts from \cite[Section 3]{Nick} and \cite[Chapter 4]{local}.
 
 \begin{lemma}\label{lem2.9} $\,$
 
 \begin{enumerate}
 \item{$p^{*}, L^{2}$  preserve and reflect both local weak and local Joyal equivalences.}
 \item{A map of sheaves of quasi-categories \textnormal{(}resp. Kan complexes\textnormal{)} on a complete Boolean algebra $\mathscr{B}$ is a local Joyal \textnormal{(}resp. weak\textnormal{)} equivalence iff it is a sectionwise Joyal \textnormal{(}resp. weak\textnormal{)} equivalence.}
 \item{A map $f: X \rightarrow Y$ of presheaves of quasi-categories \textnormal{(}Kan complexes\textnormal{)} is a local Joyal \textnormal{(}weak\textnormal{)} equivalence iff $p^{*}L^{2}(f)$ is a sectionwise Joyal \textnormal{(}weak\textnormal{)} equivalence. }
  \item{Sectionwise Joyal \textnormal{(}weak\textnormal{)} equivalences are local Joyal \textnormal{(}weak\textnormal{)} equivalences.}
  \item{Local trivial fibrations are both local Joyal and local weak equivalences.}
  \end{enumerate}
 \end{lemma}
 
 The preceding lemma is useful because it allows us to replace stalkwise reasoning with sectionwise reasoning. This makes the local model structures easier to work with.

 We write $\mathcal{L}_{\mathrm{inj}}(X), \mathcal{L}_{\mathrm{Joyal}}(X)$, respectively, for the fibrant replacement of a simplicial presheaf in the Jardine and local Joyal model structures. 
 
\begin{definition}\label{def2.10}
We say that a presheaf of Kan complexes (quasi-categories) $X$ satisfies injective descent (quasi-injective descent) iff 
$X \rightarrow \mathcal{L}_{\mathrm{inj}}(X)$ ($X \rightarrow \mathcal{L}_{\mathrm{Joyal}}(X)$) is a sectionwise weak (Joyal) equivalence. 
\end{definition}

Given a presheaf of quasi-categories $X$, we write $J(X)$ for the presheaf obtained by applying $J$ to $X$ sectionwise. The following theorem is \cite[Theorem 4.7]{Nick2}:

\begin{theorem}\label{thm2.11}
 Suppose that $X$ is a presheaf of quasi-categories. Then $X$ satisfies quasi-injective descent iff for each $n \in \mathbb{N}$,
$J(X^{\Delta^{n}})$ satisfies injective descent. In particular, if $X$ is a presheaf of Kan complexes, then it satisfies quasi-injective descent iff it satisfies injective descent. 
\end{theorem}

The remainder of the section will be devoted to relating the notion of injective descent to the notion of descent along hypercovers. 

Suppose that $U \in \mathrm{Ob}( \mathscr{C})$. Then we write $\bar{U}$ for the representable simplicial presheaf $\mathrm{hom}(-, U)$. 

\begin{definition}\label{def2.12}
A \textbf{hypercover} of $U$ is a local trivial fibration $V \rightarrow \bar{U}$, such that each $V_{n}$ is a coproduct of representables. 

If $X$ is a presheaf of Kan complexes (quasi-categories), we say that $X$ \textbf{satisfies descent} with respect to a hypercover $V \rightarrow \bar{U}$ iff 
$$
X(U) \rightarrow \mathrm{holim}_{n \in \Delta^{*}} X(V_{n})
$$
is a weak (Joyal) equivalence, where the homotopy limit is taken in the standard (Joyal) model structure on simplicial sets.  
\end{definition}

\begin{example}\label{exam2.13}
Let  $\mathcal{V}  = \{ U_{i} \rightarrow U\}_{i \in I}$ be a cover. Then there is a cosimplicial object $\tilde{C}(\mathcal{V})$ in $\mathscr{C}$, which is given in degree n by the formula
$$
\coprod_{\alpha_{0}, \cdots \alpha_{n} \in I} U_{\alpha_{0}} \times_{U} U_{\alpha_{2}} \cdots \times_{U} U_{\alpha_{n}}.
$$
This cosimplicial object represents a hypercover, known as the \textbf{Cech nerve}. 
If $\mathcal{V}$ is generated by a single morphism $\phi$, we write $\tilde{C}(\phi)$ for $\tilde{C}(\mathcal{V})$. We say that a presheaf of Kan complexes (quasi-categories)  \textbf{satisfies descent} with respect to the cover $\mathcal{V}$ iff it satisfies descent with respect to the hypercover $\tilde{C}(\mathcal{V})$. 

\end{example}

\begin{lemma}\label{lem2.14}
Suppose that $X$ is a presheaf of quasi-categories. Then $X$ satisfies descent with respect to a hypercover $V \rightarrow \bar{U}$ iff $J(X^{\Delta^{n}})$ satisfies descent with respect to $V \rightarrow \bar{U}$ for each $n \in \mathbb{N}$.
\end{lemma}

\begin{proof}
As shown in \cite[Theorem 1.19]{JT1} and \cite[Corollary 3.14]{Nick2}, there is a Quillen functor
$$
k_{!} : \mathrm{sSet} \rightleftarrows \mathrm{sSet} : k^{!}
$$
from the standard model structure to the Joyal model structure. There is also a natural trivial fibration $k^{!} \rightarrow J$ (\cite[Lemma 3.11]{Nick2}). We have a chain of weak equivalences, natural in $V$:
\begin{align*}
J((\mathrm{holim}_{n \in \Delta^{*}} X(V_{n}))^{\Delta^{n}}) \simeq J(\mathrm{holim}_{n \in \Delta^{*}} (X(V_{n})^{\Delta^{n}}) ) \simeq \\
k^{!}(\mathrm{holim}_{n \in \Delta^{*}} (X(V_{n})^{\Delta^{n}}) ) \simeq \mathrm{holim}_{n \in \Delta^{*}} k^{!}(X(V_{n})^{\Delta^{n}} ) \simeq \\
 \mathrm{holim}_{n \in \Delta^{*}} J(X(V_{n})^{\Delta^{n}} )
\end{align*}
The first weak equivalence follows from the fact that $(-)^{\Delta^{n}}$ is a right Quillen functor for the Joyal model structure (see \cite[Corollary 2.2.5.4]{Lurie}), and hence commutes with homotopy limits.
The second and fourth follow from \cite[Lemma 3.11]{Nick2}. The third follows from the fact that $k^{!}$ is a right Quillen functor, and thus commutes with homotopy limits. 

\end{proof}

Recall that there is a \textbf{global injective (Kan) model structure} on $s\textbf{Pre}(\mathscr{C})$ in which the weak equivalences are sectionwise weak equivalences and the cofibrations are monomorphisms. 

\begin{theorem}\label{thm2.15}\textnormal{\cite[Theorem 1.1]{Dugger-Isaksen}}.
A presheaf of Kan complexes is a fibrant object in the Jardine model structures iff
\begin{enumerate}
\item{It is fibrant for the global injective model structure.}
\item{It satisfies descent with respect to each hypercover.}
\end{enumerate}
\end{theorem}

\begin{corollary}\label{cor2.16}
A presheaf of Kan complexes \textnormal{(}quasi-categories\textnormal{)} satisfies injective \textnormal{(}quasi-injective\textnormal{)} descent iff it satisfies descent with respect to each hypercover.
\end{corollary}

\begin{proof}
The case of a presheaf of Kan complexes is immediate from Theorem~\ref{thm2.15}. The case of presheaves of quasi-categories then follows from Theorem~\ref{thm2.11} and Lemma~\ref{lem2.14}.
\end{proof}

\begin{remark}\label{rmk2.17}
In \cite{Lurie} a presheaf of Kan complexes is said to \textbf{satisfy descent} iff it satisfies descent for each cover, and a presheaf of Kan complexes is said to satisfy \textbf{hyperdescent} iff it satisfies descent with respect to each hypercover. An example of a presheaf of Kan complexes that satisfies descent but not hyperdescent is given in \cite[6.5.4.8]{Lurie}.

It is important to keep in mind this distinction, as two of the main results of the paper (Theorem~\ref{thm4.10} and Theorem~\ref{thm4.12}) shows that hyperdescent for a presheaf of quasi-categories is equivalent to descent for all coverings and hyperdescent for all mapping space presheaves. 
\end{remark}

\section{The Efffective Descent Condition for a Presheaf of Quasi-Categories}\label{sec5}

In this section we introduce a quasi-categorical analogue of the effective descent condition of \cite[Definition 10.1]{Simpson}, which is somewhat weaker than descent with respect to a cover in the sense of Example~\ref{exam2.13}. This condition is needed to formulate one of the major results of the paper (Theorem~\ref{thm4.10}). 

This section is devoted to proving a number of characterizations of effective descent to be used in Section~\ref{sec6}. Of particular importance is the characterization of effective descent in terms of a $\mathrm{holim}$ indexed by elements of the cover (Lemma~\ref{lem3.3}). In the case the cover is generated by a single element, there is a final inclusion $\Delta \subseteq R$, so we are reduced to taking $\mathrm{holim}$ over a cosimplicial diagram (Example~\ref{exam3.6}). This allows us to deduce Theorem~\ref{thm4.12} from Theorem~\ref{thm4.10} in Section~\ref{sec6}.
\\

There is a \textbf{global injective (Joyal) model structure} on $s\textbf{Pre}(\mathscr{C})$, in which the cofibrations and weak equivalences are, respectively, sectionwise cofibrations and sectionwise weak equivalences for the Joyal model structure.  

Recall that if $X$ is a simplicial presheaf and $U \in \mathrm{Ob}(\mathscr{C})$, then we write $X|_{U}$ for the presheaf obtained by composing the forgetful functor from the slice category $(\mathscr{C}/U)^{\mathrm{op}} \rightarrow \mathscr{C}^{\mathrm{op}}$ with $X$. If $R$ is a covering of some object $U$, we can regard it as a subcategory of the slice category $\mathscr{C}/U$ and write $X|_{R}$ for the restriction of $X|_{U}$ to $R^{\mathrm{op}}$.  

\begin{definition}\label{def3.1}
Let $X$ be a presheaf of quasi-categories. Then $X$ is said to satisfy \textbf{effective descent} with respect to a covering sieve $R$ of an object $U$ if and only if 
$[*|_{U}, X]_{q} \rightarrow [*|_{R}, X]_{q}$ is surjective, where $[\, , \,]_{q}$ denotes maps in the homotopy category of the global injective Joyal model structure. $X$ is said to satisfy \textbf{effective descent} if and only if it satisfies effective descent with respect to each covering sieve of each object. 

\end{definition}

\begin{remark}\label{rmk3.2}
Note that the effective descent condition is invariant under sectionwise Joyal equivalence. 
\end{remark}

\begin{lemma}\label{lem3.3}
Let $X$ be a presheaf of quasi-categories. The following are equivalent:

\begin{enumerate}
\item{X satisfies effective descent with respect to a cover $R$ of $U$.}
\item{$\pi_{I}(*|_{U}, X') \rightarrow \pi_{I}(*|_{R}, X')$ is surjective for some global injective fibrant replacement $X'$ of $X$. Here, $I = B \pi (\Delta^{1})$ is a constant simplicial presheaf, and $\pi_{I}$ means $I$-homotopy classes of maps.}
\item{ $X'(U) \rightarrow (\mathrm{lim}_{\phi : V \rightarrow U, \phi \in R^{\mathrm{op}} } X'(V))$ is essentially surjective.}
\item{$X(U) \rightarrow \mathrm{holim}_{\phi : V \rightarrow U, \phi \in R^{\mathrm{op}} } X(V)$ is essentially surjective. Here, $\mathrm{holim}$ means the homotopy limit with respect to the Joyal model structure.}
\end{enumerate}

\end{lemma}

\begin{proof}

($1. \iff 2.$)
By Remark~\ref{rmk3.2}, $X$ satisfies effective descent iff $X'$ does. $I$ is an interval object for the global injective Joyal model structure, and everything is cofibrant, so we have identifications (natural in $A$) 
$$[A, X] = [A, X'] = \pi_{I}(A, X'). $$
The result follows.
\\

\noindent
$(2. \iff 4)$. We have identifications (natural in subcategories $W$ of $\mathscr{C}/U$):
$$
\pi_{I}(*|_{W}, X') = \pi_{I}  (*, (\mathrm{lim}_{\phi : V \rightarrow U, \phi \in W^{\mathrm{op}} } X'(V))) = \pi_{0}J((\mathrm{lim}_{\phi : V \rightarrow U, \phi \in W^{\mathrm{op}} } X'(V))).
$$ 
Thus, $[*|_{U}, X] \rightarrow [*|_{R}, X]$ can be naturally identified with 
 $$
\pi_{0}(JX'(U)) \cong \pi_{0}J (\mathrm{lim}_{\phi : V \rightarrow U, \phi \in (\mathscr{C}/U)^{\mathrm{op}} } X'(V)) \rightarrow \pi_{0}J( \mathrm{lim}_{\phi : V \rightarrow U, \phi \in R^{\mathrm{op}} } X'(V)).  
$$   
      
Now, $(3 \iff 4)$ follows immediately from the definition of homotopy limits. 
\end{proof}

\begin{lemma}\label{lem3.4}
Suppose that $X$ is fibrant for the global injective Joyal model structure on $s\mathbf{Pre}(\mathscr{C})$. Then $X$ satisfies effective descent with respect to a covering $R$ of $U$ if and only if it has the right lifting property with respect to $*|_{R} \rightarrow *|_{U}$.
\end{lemma}

\begin{proof}
Sufficiency follows from 2. of Lemma~\ref{lem3.3}. We prove necessity. 
Suppose that we have a diagram
$$
\xymatrix
{
\, *|_{R} \ar[r]^{s} \ar[d] & X\\
\, *|_{U} \ar@{.>}[ur] &
}
$$
By hypothesis, we can choose $t : *|_{U} \rightarrow X$ and a homotopy $h: I \times *|_{R} \rightarrow X$ between $s$ and $t|_{R}$. We can find a lift
$$
\xymatrix
{
I|_{R} \times_{*|_{R}} *|_{U} \ar[d] \ar[r]^>>>>>{(h, t)} & X \\
I|_{U} \ar@{.>}[ur]_{\phi} & 
}
$$ 
since the vertical map is a sectionwise Joyal equivalence. The map $\phi |_{0}$ gives the required lift in our first diagram. 
\end{proof}

We say that a subcategory $C \subseteq R$ \textbf{generates} the covering sieve iff this inclusion is final in the sense of \cite[Section IX.3]{CWM}.

\begin{lemma}\label{lem3.5}
Suppose that $C$ generates $R$. Then $X$ satisfies effective descent with respect to a cover $R$ of $U$ iff 
$$
X(U) \rightarrow  (\mathrm{holim}_{\phi : V \rightarrow U, \phi \in C^{\mathrm{op}} } X(V))
$$
is essentially surjective.
\end{lemma}

\begin{proof}
The fact that the inclusion $C \subseteq R$ is a final functor implies that the induced map
$$
 (\mathrm{lim}_{\phi : V \rightarrow U, \phi \in C^{\mathrm{op}} } X(V)) \rightarrow (\mathrm{lim}_{\phi : V \rightarrow U, \phi \in R^{\mathrm{op}} } X(V))
$$  
is an isomorphism. 
Since we can assume that $X$ is global injective fibrant, the result follows from 3. of Lemma~\ref{lem3.3}. 
\end{proof}

 \begin{example}\label{exam3.6}
 Suppose that $R$ is generated by a single morphism $\phi : V \rightarrow U$. Then the functor 
 $$
S : \Delta^{\mathrm{op}} \rightarrow R, \qquad [\textbf{n}] \mapsto \tilde{C}(\phi)_{n}
$$
 is final (see the proof of \cite[Corollaire 15.7]{Simpson}), where $\tilde{C}(\phi)$ is the Cech nerve of Example~\ref{exam2.13}. Thus, the effective descent condition with respect to $R$ is equivalent to 
 $$
X(U) \rightarrow \mathrm{holim}_{n \in \Delta^{*}} X(\tilde{C}(\phi)_{n})
$$
being essentially surjective. In this case, the effective descent condition for $R$ coincides with the notion of descent with respect to the hypercover associated to $\tilde{C}(\phi)$ studied at the end of Section~\ref{sec4}.   
 
 \end{example}

\begin{lemma}\label{lem3.7}
Suppose that $C$ is a category such that:
\begin{enumerate}
\item{$C$ has $\alpha$-bounded coproducts and pullbacks.}
\item{coproducts and pullbacks commute.}
\end{enumerate}
 Then there is a Grothendieck topology on $C$ \textnormal{(}the \textbf{connected components topology}\textnormal{)}, whose covers are generated by the inclusions 

$$
\left\{ U_{i} \rightarrow \coprod_{j \in I} U_{j} \right\}_{i \in I}
$$
such that $|I| < \alpha$.
\end{lemma}

\begin{proof}
We verify the axioms of Definition 1 on \cite[pg.~110]{MM}. The stability under pullbacks axiom follows from condition 2. on $C$. The transitivity axiom is similarly easy to check.  
\end{proof}

The following is immediate from Lemma~\ref{lem3.5}: 

\begin{corollary}\label{cor3.8}
Suppose that $\mathscr{C}$ has the Grothendieck topology of \textnormal{Lemma~\ref{lem3.7}} and $X$ is a presheaf of quasi-categories on $\mathscr{C}$. Then $X$ satisfies effective descent iff 
$$
X\left(\coprod U_{\alpha}
\right) \rightarrow \sideset{}{^{h}}\prod_{\alpha \in I}  X(U_{\alpha})
$$
 is essentially surjective for each $\alpha$-bounded collection $\{ U_{\alpha} \}_{\alpha \in I}$. Here, the superscript $h$ means `homotopy product', i.e., a homotopy limit. 
\end{corollary}

\section{Descent for Presheaves of Quasi-Categories}\label{sec6}

 The purpose of this section is to study quasi-injective descent for a presheaf of quasi-categories. We will first introduce mapping space presheaves, and show that the interact well with the local Joyal model structure. 
 Our main results are Theorem~\ref{thm4.10} and Theorem~\ref{thm4.12}, which are quasi-categorical analogues of \cite[Theorems 10.2, 19.4]{Simpson}. Essentially, these theorems (in the language of Remark~\ref{rmk2.17}) say that a presheaf of quasi-categories satisfies hyperdescent iff it satisfies descent with respect to all covers and all mapping space presheaves satisfy hyperdescent.

\begin{definition}\label{def4.1}
Suppose that $X, Y$ are simplicial presheaves. We define their \textbf{join}, $X * Y$, to be the simplicial presheaf obtained by applying the usual join operation sectionwise. Suppose that $f: K \rightarrow X$ is a map of simplicial presheaves with $K$ constant. Then we can form a simplicial presheaf
$X_{/f}$
such that $X_{/f}(U) = X(U)_{/f}$. 

Suppose that $X$ is a presheaf of quasi-categories and let $x, y : * \rightarrow X$ be global sections. Then the mapping space presheaf $\mathrm{Map}_{X}(x, y) $ is defined to be the pullback
$$
\xymatrix
{
\mathrm{Map}_{X}(x, y) \ar[r] \ar[d] & X_{/y} \ar[d] \\
\, * \ar[r]_{x} & X
}
$$
\end{definition}

\begin{lemma}\label{lem4.2} Let $A, B$ be presheaves of simplicial sets. Then 
$$p^{*}L^{2}(A * B) \cong L^{2}(p^{*}L^{2}A * p^{*}L^{2}B).$$ 

\end{lemma}

\begin{proof}
Since $p^{*}L^{2}$ commutes with colimits and finite limits, we have the following:
\begin{align*}
p^{*}L^{2}(A* B) \cong L^{2} ( (p^{*}L^{2}A_{n}) \cup (p^{*}L^{2}B_{n}) \bigcup_{i + j = n-1} (p^{*}L^{2}A_{i} \times p^{*}L^{2}B_{j})) \\
 \cong L^{2}(p^{*}L^{2}A * p^{*}L^{2}B).
\end{align*}
\end{proof}

\begin{lemma}\label{lem4.3}
Let $K$ be a finite simplicial set \textnormal{(}i.e., $K$ has finitely many non-degenerate simplices\textnormal{)}, regarded as a constant simplicial presheaf. Let $X$ be a simplicial presheaf and $f: K \rightarrow X$ a map of simplicial presheaves, then 
$$
p^{*}L^{2}(X_{/ f}) \cong p^{*}L^{2}(X)_{/ p^{*}L^{2}(f)}.
$$

\end{lemma}

\begin{proof}
The $n$-simplices of $X_{/f}$ can be described as the pullback:
$$
\xymatrix
{
(X_{/f})_{n} \ar[d] \ar[r] & \mathrm{hom}(K * \Delta^{n}, X) \ar[d] \\
\, * \ar[r]_>>>>>>>>{f} & \mathrm{hom}(K, X).
}
$$
The fact that $p^{*}L^{2}$ preserves pullbacks and Lemma~\ref{lem2.5} imply that $p^{*}L^{2}(X_{/f})_{n}$ is isomorphic to the pullback of
$$
\xymatrix
{
  & \mathrm{hom}(K * \Delta^{n}, p^{*}L^{2}X) \ar[d] \\
\, * \ar[r]_>>>>>>>>>{p^{*}L^{2}f} & \mathrm{hom}(K, p^{*}L^{2}X),
}
$$
which is just $(p^{*}L^{2}(X)_{/ p^{*}L^{2}(f)})_{n}$ .
\end{proof}

\begin{lemma}\label{lem4.4}
Suppose that $f: X \rightarrow Y$ is a local Joyal equivalence and $A$ is a simplicial presheaf. Then $X * A \rightarrow Y * A$ is a local Joyal equivalence. 
\end{lemma}

\begin{proof}
Consider the commutative diagram 
$$
\xymatrix
{
X * A \ar[r] \ar[d] & \mathcal{S}_{\mathrm{Joya}l}(X) *  \mathcal{S}_{\mathrm{Joyal}}(A) \ar[d] \\
Y  * A \ar[r] & \mathcal{S}_{\mathrm{Joyal}}(Y) * \mathcal{S}_{\mathrm{Joyal}}(A)
}
$$
By \cite[Corollary 4.2.1.3]{Lurie}, the horizontal maps are sectionwise Joyal equivalences, and hence local Joyal equivalences by Lemma~\ref{lem2.9}. 
Thus, by the two out of three property, it suffices to show that $\mathcal{S}_{\mathrm{Joyal}}(X) * \mathcal{S}_{\mathrm{Joyal}}(A) \rightarrow \mathcal{S}_{\mathrm{Joyal}}(Y) * \mathcal{S}_{\mathrm{Joyal}}(A)$ is a local Joyal equivalence. We have reduced to the case that $A, X, Y$ are presheaves of quasi-categories. By \cite[Corollary 4.2.1.3]{Lurie} and 3. of Lemma~\ref{lem2.9}, the map
$$
p^{*}L^{2}(f) * \mathrm{id} : p^{*}L^{2} X * p^{*}L^{2} A \rightarrow p^{*}L^{2}Y * p^{*}L^{2}A
$$
is a sectionwise Joyal equivalence. Thus $L^{2}(p^{*}L^{2}(f) * \mathrm{id}) \cong p^{*}L^{2}(f * \mathrm{id})$ is a local Joyal equivalence by 1. of Lemma~\ref{lem2.9}.
The result follows from that fact that $p^{*}L^{2}$ reflects local Joyal equivalences (Lemma~\ref{lem2.9}).
\end{proof}

\begin{lemma}\label{lem4.5}
Suppose that $f: X \rightarrow Y$ is a presheaf of quasi-categories. Let $K$ be a finite simplicial set, regarded as a constant simplicial presheaf, and let $s : K \rightarrow X$ be a map. Then $X_{/s} \rightarrow Y_{/f \circ s}$ is a local Joyal equivalence. 
\end{lemma}

\begin{proof}
Since $p^{*}L^{2}$ preserves and reflects local Joyal equivalences, it suffices to show that $p^{*}L^{2}(X_{/s}) \rightarrow p^{*}L^{2}(Y_{/f \circ s})$ is a sectionwise Joyal equivalence. But this is naturally isomorphic to the map 
$$
p^{*}L^{2}(X)_{/p^{*}L^{2}(s)} \rightarrow  p^{*}L^{2}(X)_{/p^{*}L^{2}(f) \circ p^{*}L^{2}(s)}.
$$
$p^{*}L^{2}(f)$ is a sectionwise Joyal equivalence (see Lemma~\ref{lem2.9}). The map in the preceding equation is then a sectionwise Joyal equivalence by the discussion of \cite[pg. 241]{Lurie}. 
\end{proof}

\begin{corollary}\label{cor4.6}
Let $X$ is quasi-injective fibrant simplicial presheaf. Then  $ \forall \, x, y \in X(U), U \in \mathrm{Ob}(\mathscr{C})$,  $\mathrm{Map}_{X|_{U}}(x, y)$ satisfies injective descent.
\end{corollary}

\begin{proof}
First, we will show that $\mathrm{Map}_{X|_{U}}(x, y)$ is quasi-injective fibrant. 
One wants to solve a lifting problem
$$
\xymatrix{
A \ar[r] \ar[d] & \mathrm{Map}_{X|_{U}}(x, y) \\
B   \ar@{.>}[ur]& 
}
$$
where $A \rightarrow B$ is a local Joyal equivalence and a monomorphism.
By adjunction, this is equivalent to a lifting problem
$$
\xymatrix
{
(A * \Delta^{0}) \coprod_{A} B \ar[d] \ar[r]^>>>>>>{m} & X \\
B * \Delta^{0} \ar@{.>}[ur]
}
$$
where $m|_{B}$ factors through the inclusion $*|_{U} \xrightarrow{x} X$ and $\Delta^{0}$ factors through $y$. But the vertical map is a local Joyal equivalence and a monomorphism, so the lift exists.  

Now, $\mathrm{Map}_{X|_{U}}(x, y)$ satisfies quasi-injective descent.  Because the mapping spaces in a quasi-category are Kan complexes, $J\mathrm{Map}_{X|_{U}}(x, y) = \mathrm{Map}_{X|_{U}}(x, y)$ and $\mathrm{Map}_{X|_{U}}(x, y)$  satisfies injective descent by Theorem~\ref{thm2.11}. 
\end{proof}

\begin{lemma}\label{lem4.7}
Suppose that $X$ is a presheaf of quasi-categories and let $i: X \rightarrow \mathcal{L}_{\mathrm{Joyal}}(X)$ be a quasi-injective fibrant replacement. Then $i$ is fully faithful in each section iff $\forall \, U \in \mathrm{Ob}(\mathscr{C}), x, y \in X(U)$, $\mathrm{Map}_{X|_{U}}(x, y)$ satisfies injective descent.
\end{lemma}

\begin{proof}
Let $x, y \in X(U)$.
Note that we have a diagram
$$
\xymatrix
{
J((X|_{U})_{/y}) \ar[r] \ar[d] & J(X|_{U}) \ar[r] \ar[d] & \ar[l]_{x} * \ar[d] \\
J((\mathcal{L}_{\mathrm{Joyal}}(X)|_{U})_{/i(y)}) \ar[r] & J(\mathcal{L}_{\mathrm{Joyal}}(X)|_{U}) & \ar[l]^>>>>>>{i(x)} * 
}
$$
The left vertical map is a local weak equivalence by Lemma~\ref{lem4.5} and the fact that $J$ sends local Joyal equivalences to local weak equivalences (\cite[Lemma 4.5]{Nick2}). 
The left horizontal maps are obtained by applying $J$ to a sectionwise right fibration (see \cite[Corollary 2.1.2.2]{Lurie}).
Since right fibrations of quasi-categories are quasi-fibrations (\cite[Proposition 4.10]{Joyal-quasi-cat}) and $J$ sends quasi-fibrations to Kan fibrations (\cite[Proposition 4.27]{Joyal-quasi-cat}), the left horizontal maps in the diagrams are sectionwise Kan fibrations. The pullbacks of each of the two rows are thus both homotopy cartesian diagrams for the Jardine model structure by \cite[Lemma 5.20]{local}.
The properness of the Jardine model structure and \cite[Lemma II.8.19]{GJ2} imply that we have a local weak equivalence
 $$J\mathrm{Map}_{X|_{U}}(x, y) \rightarrow J\mathrm{Map}_{\mathcal{L}_{\mathrm{Joyal}}(X)|_{U}}(i(x), i(y)) $$ 
The functor $J$ is the identity on Kan complexes, so we have a local weak equivalence
$$
\mathrm{Map}_{X|_{U}}(x, y) \rightarrow \mathrm{Map}_{\mathcal{L}_{\mathrm{Joyal}}(X)|_{U}}(i(x), i(y)).
$$
The object $\mathrm{Map}_{\mathcal{L}_{\mathrm{Joyal}}(X)|_{U}}(i(x), i(y))$ satisfies injective descent by Corollary~\ref{cor4.6}. Thus, $\mathrm{Map}_{X|_{U}}(x, y) \rightarrow \mathrm{Map}_{\mathcal{L}_{\mathrm{Joyal}}(X)|_{U}}(i(x), i(y))$ is a sectionwise weak equivalence iff $\mathrm{Map}_{X|_{U}}(x, y)$ satisfies injective descent. 
\end{proof}

\begin{lemma}\label{lem4.8}

The fibrant replacement $i: X \rightarrow \mathcal{L}_{\mathrm{Joyal}}(X)$ is essentially surjective in sections iff $J(X) \rightarrow \mathcal{L}_{\mathrm{inj}}(JX)$ is essentially surjective in sections. 
\end{lemma}

\begin{proof}
Let $U \in \mathrm{Ob}(\mathscr{C})$. By Lemma~\ref{lem1.5}, it suffices to show that $\pi_{0}(JX(U)) \rightarrow \pi_{0}(J \mathcal{L}_{\mathrm{Joyal}}(X)(U))$ is a surjection iff $\pi_{0}(JX(U)) \rightarrow \pi_{0}(\mathcal{L}_{\mathrm{inj}}(JX)(U))$ is. 
There is a diagram: 
$$
\xymatrix
{
JX \ar[r] \ar[d] & \mathcal{L}_{\mathrm{inj}}(JX) \\
J \mathcal{L}_{\mathrm{Joyal}}(X) & 
}
$$ 
Since $J$ sends local Joyal equivalences of presheaves of quasi-categories to local weak equivalences by \cite[Lemma 4.5]{Nick2}, we can produce a lifting
$$
\xymatrix
{
J(X) \ar[r] \ar[d] & \mathcal{L}_{\mathrm{inj}}(JX) \\
J \mathcal{L}_{\mathrm{Joyal}}(X) \ar@{.>}[ur]_{\phi} & 
}
$$ 
The diagonal map is a local weak equivalence by two out of three. A weak equivalence whose source satisfies injective descent and whose target is injective fibrant is a sectionwise weak equivalence (\cite[Corollary 5.15]{local}). But $J \mathcal{L}_{\mathrm{Joyal}}(X)$ satisfies injective descent by Theorem~\ref{thm2.11}, so we conclude that $\phi$ is in fact a sectionwise weak equivalence. The required statement follows. 

\end{proof}

\begin{theorem}\label{thm4.9} Let $X$ be a presheaf of quasi-categories. Then X satisfies quasi-injective descent iff 
\begin{enumerate}
\item{For each $x, y \in X(U), U \in \mathrm{Ob}(\mathscr{C})$, $\mathrm{Map}_{X|_{U}}(x, y)$ satisfies injective descent.}
\item{$JX$ satisfies injective descent.}
\end{enumerate}
\end{theorem}

\begin{proof}
 If $X$ satisfies quasi-injective descent, then these conditions hold by Theorem~\ref{thm2.11} and Lemma~\ref{lem4.7}. On the other hand, suppose the conditions above hold and let $i : X \rightarrow \mathcal{L}_{\mathrm{Joyal}}(X)$ be a fibrant replacement. $X$ satisfies quasi-injective descent iff this map is fully faithful and essentially surjective in sections (Theorem~\ref{thm1.6}). This is indeed the case by Lemmas~\ref{lem4.7} and~\ref{lem4.8}.  
  
\end{proof}

\begin{theorem}\label{thm4.10}
Let $X$ be a presheaf of quasi-categories. Then $X$ satisfies quasi-injective descent iff 
\begin{enumerate}
\item{For each $U \in \mathrm{Ob}(\mathscr{C}), x, y \in X(U)$, $\mathrm{Map}_{X|_{U}}(x, y)$ satisfies injective descent.}
\item{X satisfies the effective descent condition with respect to all coverings $R$ of objects $U$.}
\end{enumerate}
\end{theorem}

\begin{proof}
Suppose that $X$ satisfies quasi-injective descent. Then Condition (1) is satisfied by Corollary~\ref{cor4.6}. To show that the effective descent condition is satisfied, note that the map $X \rightarrow\mathcal{L}_{\mathrm{Joyal}}(X)$ is a fibrant replacement for the global injective Joyal model structure because $X$ satisfies quasi-injective descent. Given a covering $R$ of an object $U$, $*|_{R} \rightarrow *|_{U}$ is a local Joyal equivalence and we can solve lifting problems
$$
\xymatrix
{
\, *|_{R} \ar[r] \ar[d] & \mathcal{L}_{\mathrm{Joyal}}(X) \\
\, *|_{U} \ar@{.>}[ur] &
}
$$
Thus, the result follows from Remark~\ref{rmk3.2} and Lemma~\ref{lem3.4}. 

Now, assume the two conditions in the statement of the theorem hold. 
By Theorem~\ref{thm1.6}, Lemma~\ref{lem4.7} and the assumption on mapping space presheaves, it suffices to show that $X \rightarrow \mathcal{L}_{\mathrm{Joyal}}(X)$ is essentially surjective in sections.  Factor $X \rightarrow B \xrightarrow{g}  \mathcal{L}_{\mathrm{Joyal}}(X)$, where $g$ is an fibration for the global injective Joyal model structure and $X \rightarrow B$ is a sectionwise Joyal equivalence. 

It now suffices to show that $g$ is essentially surjective in sections.
Let $a :*|_{U} \rightarrow \mathcal{L}_{\mathrm{Joyal}}(X)$ be a point. Form the pullback
$$
\xymatrix
{
C \ar[r]^{l} \ar[d]_{k} & B \ar[d]^{g} \\
\, *|_{U} \ar[r]_>>>>>{a} &  \mathcal{L}_{\mathrm{Joyal}}(X)
}
$$
The map $g$ is a sectionwise quasi-fibration and a local Joyal equivalence, and hence a local trivial fibration by \cite[Lemma 3.15]{Nick}. Thus, 
$k : C \rightarrow *|_{U}$ is a local trivial fibration. Thus is has the local right lifting property (Definition~\ref{def2.1}) with respect to the map $\emptyset \rightarrow \Delta^{0}$. Thus, there exists a covering   $R$ of $U$ such that for each $\phi : V \rightarrow U$ in $R$, $C(V)$ is non-empty. We conclude from this that $k|_{R} : C|_{R} \rightarrow *|_{R}$ is surjective. 
Let $S \subseteq  \mathcal{L}_{\mathrm{Joyal}}(X)$ be the full subpresheaf of quasi-categories (i.e. in sections; see \cite[1.2.11]{Lurie})
such that 
$$
\mathrm{Ob}(S) = \mathrm{im}(\mathrm{Ob}(g)).
$$
One can factor the map $g$ as $B \xrightarrow{g'} S \subseteq \mathcal{L}_{\mathrm{Joyal}}(X)$.

The map $g'$ is a fibration for the global injective model structure, since $g'$ can be expressed as a pullback

$$
\xymatrix
{
B \ar[r]^{id_{B}} \ar[d]^{g'} & B\ar[d]^{g} \\
S \ar[r] & \mathcal{L}_{\mathrm{Joyal}}(X')
}
$$

$g'$ is essentially surjective in sections by construction. It is fully faithful in sections by \cite[Remark 1.2.2.4]{Lurie} and the fact that $g$ is fully faithful in sections. Thus, it is a sectionwise Joyal equivalence by Theorem~\ref{thm1.6}. Thus it is a trivial quasi-injective fibration, since trivial fibrations for the local Joyal and global injective Joyal model structures coincide. $k|_{R} : C|_{R} \rightarrow *|_{R}$ is isomorphic to the pullback
$$
\xymatrix
{
C|_{R} \ar[r]^{k|_{R}} \ar[d]_{l|_{R}} & \ar[d]^{a|_{R}} *|_{R} \\
B \ar[r]_{g'} & S
}
$$
so that $k|_{R} : C|_{R} \rightarrow *|_{R}$ is a trivial quasi-injective fibration. Because every object is cofibrant in the local Joyal model structure we can find a lift
$$
\xymatrix
{
\emptyset \ar[r] \ar[d] & C|_{R} \ar[d]^{k|_{R}} \\
\, *|_{R} \ar[r]_{id} \ar@{.>}[ur]_{s} & \, *|_{R}
}
$$

 That is, $k|_{R}$ has a section $s: *|_{R} \rightarrow C|_{R}$. The morphism 
$$
l |_{R} \circ s : *|_{R} \rightarrow B
$$ 
 extends to a map
$$
r: *|_{U} \rightarrow B
$$
by the effective descent hypothesis, Lemma~\ref{lem3.4} and the fact that $B$ is global quasi-injective fibrant replacement of $X$. Now, there exists a lifting 
$$\xymatrixcolsep{6pc}
\xymatrix
{
\, (*|_{U} \coprod \, *|_{U}) \coprod_{\, (*|_{R} \coprod \, *|_{R} )} (I|_{R}) \ar[r]^>>>>>>>>>>>>>>{(g \circ r, a), a|_{R} \circ t)} \ar[d] &  \mathcal{L}_{\mathrm{Joyal}}(X) \\
I|_{U} \ar@{.>}[ur] &
}
$$ 
where $t : I|_{R} \rightarrow *|_{R}$ is the terminal map in sections (the vertical map is a local Joyal equivalence).
 Thus, the image of the point $r$ is isomorphic to $a$ in $P(\mathcal{L}_{\mathrm{Joyal}}(X))$ by Theorem~\ref{thm1.0}. Since $a, U$ are arbitrary, we have proven that $g$ is essentially surjective in sections. 
\end{proof}

\begin{definition}\label{def4.11}
We say that a site $\mathscr{E}$ with fiber products has \textbf{sufficient coproducts compatible with fiber products}  if there
exists some cardinal $\alpha$ such that the following conditions hold:

\begin{enumerate}
\item{$\mathscr{E}$ has $\alpha$-bounded coproducts.}
 \item{For each covering sieve $R$, there exists a final inclusion $S \subseteq R$, such that $S$ is a covering sieve generated by an $\alpha$-bounded set of morphisms.}
 \item{Fiber products commute with $\alpha$-bounded coproducts.}
\end{enumerate}
\end{definition}

\begin{theorem}\label{thm4.12}
Assume that the underlying site $\mathscr{C}$ satisfies the hypotheses of \textnormal{Definition~\ref{def4.11}}. 
Let $X$ be a presheaf of quasi-categories such that $X(i) = *$, where $i$ is the initial object of $\mathscr{C}$. 
Suppose that $X$ satisfies the following three conditions:
\begin{enumerate}
\item{For each $U \in \mathrm{Ob}(\mathscr{C}), x, y \in X(U)$, $\mathrm{Map}_{X|_{U}}(x, y)$ satisfies injective descent.}
\item{For each morphism $\phi : V \rightarrow U$ generating a cover, $$X(U) \rightarrow \mathrm{holim}_{n \in \Delta^{*} } X(\tilde{C}(\phi)_{n})$$ is essentially surjective.}
\item{Let $\alpha$ be as in \textnormal{Definition~\ref{def4.11}}. For any $\alpha$-bounded collection of elements $\mathcal{V}$, the map induced by restriction $X(\coprod_{V \in \mathcal{V}} V) \rightarrow \prod_{V \in \mathcal{V}}^{h} X(V)$ is essentially surjective. Here, the superscript $h$ means homotopy product.}
\end{enumerate}
Then $X$ satisfies quasi-injective descent. 
\end{theorem}

\begin{proof}
These conditions are invariant under sectionwise Joyal equivalence. Thus, it suffices to assume that $X$ is global injective fibrant. 
By Theorem~\ref{thm4.10}, we want to verify the effective descent condition for each covering $R$ of an object $U$. 
Let $S \subseteq R$ be final inclusion of covering sieve, such that $S$ is generated by an $\alpha$-bounded collection $\textbf{S}$. Let $W = \coprod_{(V \rightarrow U) \in \textbf{S}} V$. Let $T$ be the sieve which is generated by inclusion $\{ V \rightarrow W  :  (V \rightarrow U) \in \textbf{S} \}$. This is a covering sieve by condition (3) of Definition~\ref{def4.11}. Note also that the morphism $W \rightarrow U$ generates a cover. 

 By Lemma~\ref{lem3.4}, we want to find liftings in the diagram
 $$
\xymatrix
 {
 \, *|_{R} \ar[d] \ar[r] & X \\
 \, *|_{W} \ar@{.>}[ur] \ar[d]  & \\
 \, *|_{U} \ar@{.>}[uur] &
 }
$$ 
The first lift exists by Corollary~\ref{cor3.8}, Lemma~\ref{lem3.4} and hypothesis 3 above.  
 The second lift exists by hypothesis 2 above, Lemma~\ref{lem3.4} and Example~\ref{exam3.6}.

\end{proof}

\begin{remark}\label{rmk4.13}
The big Zariski, etale, fppf and fpqc sites of a scheme $X$ satisfy the hypotheses of Definition~\ref{def4.11}. If $X$ is quasi-compact, then we can take $\alpha = \aleph_{0}$. 
\end{remark}

\section{Descent for Presheaves of Model Categories}\label{sec7}

In this section, we will define a simplicial localization functor $\mathfrak{B}L_{H}(-^{c})$, which associates to a model category a quasi-category. Given a (left Quillen) presheaf of model categories $M$, $\mathfrak{B}L_{H}(M^{c})$ means applying $\mathfrak{B}L_{H}(-^{c})$ to $M$ in sections. We will give a concrete interpretation of the effective descent condition for $\mathfrak{B}L_{H}(M^{c})$ in Theorem~\ref{thm5.5} in terms of 'lax sections'. This follows from a result of Bergner (Theorem~\ref{thm5.3}), which calculates the homotopy limit of a diagram of simplicial localizations of model categories. 

In Theorem~\ref{thm5.8}, we give a criterion for $\mathfrak{B}L_{H}(M^{c})$ to satisfy descent in the case that $M = \textbf{M}_{0}$ for a presheaf of simplicial model categories $\textbf{M}$. This criterion makes verifying condition (1) of Theorem~\ref{thm4.12} much easier in practice (compare our treatment of the higher stack of complexes in Section~\ref{sec8} with that of \cite{Simpson}). 

Model categories are not in general small categories, which means that one has to be careful when talking about presheaves of model categories. One way of getting around this problem is to use the technique of Grothendieck universes (see the discussion of \cite[Section 1.2.15]{Lurie}). An alternative is to replace a presheaf of model categories with a subpresheaf of small subcategories. We explain in Remark~\ref{rmk5.9} how the results of this section apply more generally to appropriate choices of subpresheaves.
\\

A \textbf{cofibration category} is a category equipped with classes of cofibrations and weak equivalences, subject to various axioms (see \cite[Definition 2.1]{Kapulkin-Szumilo}). It is a weaker model for homotopy theory than a model category. A functor between cofibration categories is called \textbf{exact} iff it preserves cofibrations and trivial cofibrations.  
  
A \textbf{diagram of model categories} is a functor from a small category to the category of model categories and left Quillen functors. A \textbf{left Quillen presheaf} is a diagram of model categories on a site.

\begin{definition}\label{def5.1}
Suppose that $M : I \rightarrow \mathrm{Cat}$ is a diagram of (model) categories. Then the \textbf{lax limit} of $M$ is the category $\mathfrak{L}_{I}(M)$ such that: 
\begin{enumerate}
\item{The objects consist of families $(x_{\alpha}, u_{\alpha, \beta}^{\theta})$ such that $x_{\alpha}$ is an object of $M(\alpha)$ and $u^{\theta}_{\alpha, \beta} $ is a morphism $M(\theta)(x_{\alpha}) \rightarrow x_{\beta}$ in $M(\beta)$, subject to the condition that 
$$
u^{\delta \theta}_{\alpha, \gamma} = u^{\delta}_{\beta, \gamma} M(\delta)( u^{\theta}_{\alpha, \beta}).
$$
Here $\alpha, \beta$ run over all objects in $I$ and $\theta : \beta \rightarrow \alpha$ runs over all morphisms of $I$. 
}

\item{A morphism $(x_{\alpha}, u_{\alpha, \beta}^{\theta}) \rightarrow (y_{\alpha}, v_{\alpha, \beta}^{\theta})$ consists of a choice of morphisms $\psi_{\alpha} : x_{\alpha} \rightarrow y_{\alpha}$ for each object $\alpha$ of $I$,
such that for each $\theta : \beta \rightarrow \alpha$ that the following diagram commutes: 
$$
\xymatrix
{
M(\theta)(x_{\alpha}) \ar[r]^{M(\theta)(\psi_{\alpha})} \ar[d]_{u_{\alpha, \beta}^{\theta}} & M(\theta)(y_{\alpha}) \ar[d]^{v_{\alpha, \beta}^{\theta}} \\
x_{\beta} \ar[r]_{\psi_{\beta}} & y_{\beta}
}
$$

}
\end{enumerate}

Suppose that $M$ is an diagram of model categories, indexed by $I$. 
We will write $\textbf{LAX}_{I}M$ for the subcategory of $\mathfrak{L}_{I}(M)$ such that the $u^{\delta}_{\alpha, \beta}$ are weak equivalences.
\end{definition}

As noted in \cite[pg. 6]{Model-Holims} if the model categories appearing in the diagram $M$ are sufficiently nice (i.e. cofibrantly generated), then $\mathfrak{L}_{I}(M)$ can be given the structure of a model structure in which the cofibrations and weak equivalences are levelwise. However, $\textbf{LAX}_{I}(M)$ cannot be given the structure of a model category because it is not closed under limits. This is because in general weak equivalences in a model category are not closed under limits. 

\begin{remark}\label{rmk5.2}
Let $\textbf{LAX}_{I}(M)^{c}$ denote the full subcategory of $\textbf{LAX}_{I}(M)$ consisting of objects $(x_{\alpha}, u_{\alpha, \beta}^{\theta})$ such that each $x_{\alpha}$ is cofibrant. Then $\textbf{LAX}_{I}(M)^{c}$  inherits the structure of a cofibration category from the model structure on $\mathfrak{L}_{I}(M)$.
\end{remark}

Given a model category $M$, we will write $M^{c}$ for the full subcategory of cofibrant objects and $M^{\circ}$ for the full subcategory of bifibrant objects. 

Suppose that $M$ is a model category, or more generally, a cofibration category. Then we write $L_{C}$ for the classifying diagram defined in \cite{Rezk}. $L_{C}(M)$ is the bisimplicial set given in horizontal degree n by
$$
B(\mathrm{we}((M^{c})^{[n]})).
$$  
Here, $(M^{c})^{[\textbf{n}]}$ is the category of functors $[\textbf{n}] \rightarrow M^{c}$, and  $\mathrm{we}((M^{c})^{[n]})$ is the subcategory whose morphisms are the levelwise weak equivalences.

In the rest of the paper we will assume that we postcompose $L_{C}$ with a Reedy fibrant replacement so that the resulting bisimplicial set is a complete Segal space (\cite[8.3]{Rezk}).  
\\

The following is \cite[Theorem 5.1]{Model-Holims}:

\begin{theorem}\label{thm5.3}
Suppose $M$ is a diagram of model categories.
There is a natural weak equivalence 
$$
L_{C}(\textnormal{\textbf{LAX}}_{I}(M)^{c}) \rightarrow \mathrm{holim}_{I}(L_{C}(M))
$$
 in the complete Segal model structure of (\cite{Rezk}). 
 
\end{theorem}

Throughout the rest of the document, we will write $L_{H}(C)$ for the hammock localization of a category with weak equivalences $C$ (see \cite[Definition 2.1]{DK-Calculating}). Whenever we write $L_{H}(C)$, we will assume that we postcompose with fibrant replacement in the Bergner model structure, so we get a fibrant simplicial category. 

The preceding theorem is proven using the following characterization of $L_{C}(X)$: 

\begin{theorem}\label{thm5.4}
Suppose that $M$ is a model category. For $x \in M$, let $\mathrm{Aut}^{h}(x)$ denote the simplicial monoid of self weak equivalences of $x$. Then $L_{C}(M^{c})$ is weakly equivalent to a bisimplicial set of the form
$$
\coprod_{[x]} B\mathrm{Aut}^{h}(x) \Leftarrow \coprod_{[\alpha] : x \rightarrow y} B\mathrm{Aut}^{h}(\alpha)  \Lleftarrow \cdots 
$$  
where the coproducts runs over weak equivalences classes of objects in $M, M^{[1]}, \cdots$.   
\end{theorem}

\begin{theorem}\label{thm5.5}
Suppose that $M$ is a diagram of model categories. Then there is a zig-zag of natural weak equivalences connecting 
$\mathrm{holim}_{I} \mathfrak{B}L_{H}(M^{c}) $ and $\mathfrak{B}L_{H}(\textnormal{\textbf{LAX}}_{I}(X)^{c})$. 
\end{theorem}

\begin{proof}

Let  $$ i_{0} : \mathrm{s^{2}Set} \rightleftarrows \mathrm{sSet} : ev_{0}$$  be the pair of adjoint functors such that $ev_{0}$ takes a bisimplicial set $X$ to $X_{0,*}$.   This is a Quillen equivalence between the Bergner and Joyal model structures (\cite[Theorem 4.11]{JT1}).
We have natural weak equivalences 
\begin{align*}
\mathfrak{B}L_{H}(\textbf{LAX}_{I}(M)^{c}) \simeq  \mathrm{ev}_{0}L_{C}(\textbf{LAX}_{I}(M)) \simeq \mathrm{ev}_{0}\mathrm{holim}_{I} L_{C}(M) \simeq 
\\ 
\mathrm{holim}_{I}  \mathrm{ev}_{0} L_{C}(M) \simeq \mathrm{holim}_{I} \mathfrak{B}L_{H}(M^{c}).
\end{align*}
 The first and the last weak equivalences follow from Remark~\ref{rmk5.2}, as well as \cite[Corollary 4.5, Remark 4.6]{Kapulkin-Szumilo}. The second follows from Theorem~\ref{thm5.3}. The third follows from the fact that $ev_{0}$ is the right adjoint of a Quillen equivalence, as well as the construction of homotopy limits.  
\end{proof}

Given a category with weak equivalences $M$, we will write $Ho(M)$ for its homotopy category (assuming it exists).

\begin{corollary}\label{cor5.6}
Suppose that $M$ is a presheaf of model categories. Then $\mathfrak{B}L_{H}(M^{c})$ satisfies the effective descent with respect to a cover $R \subseteq \mathrm{hom}(-, U)$ iff the map  
$$
\mathrm{Ho}(M(U)) \rightarrow \mathrm{Ho}(\textnormal{\textbf{LAX}}_{R^{\mathrm{op}}} (M|_{R}))
$$ 
induced by restriction is essentially surjective. 
\end{corollary}

\begin{proof}

By Theorem~\ref{thm5.5}, we want to show that the map 

$$
\mathfrak{B}L_{H}(M(U)^{c}) \rightarrow \mathfrak{B}L_{H}(\textbf{LAX}_{R^{\mathrm{op}}}(M|_{R})^{c}) 
$$
induced by restriction is an essentially surjective map of quasi-categories. 

We have isomorphisms, natural in categories with weak equivalences $Q$
$$
\pi_{0}\mathfrak{B}L_{H}(Q) \cong \pi_{0}L_{H}(Q) \simeq Ho(Q)
$$
by Theorem~\ref{thm1.9} and \cite[Proposition 3.1]{DK-Calculating}. Thus, the result follows.

\end{proof}

Suppose that $\textbf{M}$ is a presheaf of simplicial model categories on a site $\mathscr{C}$, so that $\textbf{M}_{0} = M$ is a left Quillen presheaf. Also, suppose that each restriction map $M(\phi) : M(U) \rightarrow M(V)$ preserves fibrant objects. Write $\textbf{M}^{\circ}$ for the subpresheaf such that $\textbf{M}^{\circ}(U)$ is the full subcategory (i.e. in each simplicial degree) consisting of bifibrant objects. Then we have the following: 

\begin{corollary}\label{cor5.7}
$\mathfrak{B}(\textnormal{\textbf{M}}^{\circ})$ satisfies effective descent (\ref{def3.1}) with respect to a cover $R \subseteq \mathrm{hom}(-, U)$ iff the map
$$
\mathrm{Ho}(M(U)) \rightarrow \mathrm{Ho}(\textnormal{\textbf{LAX}}_{R^{\mathrm{op}}} (M|_{R}))
$$ 
induced by restriction is essentially surjective. 
 \end{corollary}

\begin{proof}
 It follows from \cite[4.8 and 5.2]{DK-Function} that there is a zig-zag of sectionwise weak equivalences (i.e. for the Bergner model structure) connecting the presheaves of fibrant simplicial categories $L_{H}(M^{c})$ and $\textbf{M}^{\circ}$. Since we can assume that all objects involved are sectionwise fibrant, we can apply the homotopy coherent nerve functor to get a zig-zag of sectionwise Joyal equivalences connecting $\mathfrak{B}(L_{H}(M^{c}))$ and $\mathfrak{B}(\textbf{M}^{\circ})$ (Theorem~\ref{thm1.8}). Thus, we have reduced to showing effective descent for $\mathfrak{B}L_{H}(M^{c})$, so the result follows from Corollary~\ref{cor5.6}.
\end{proof}

\begin{theorem}\label{thm5.8}
Suppose that $\textnormal{\textbf{M}}$ is a presheaf of simplicial model categories. Then $\mathfrak{B}(\textnormal{\textbf{M}}^{\circ})$ satisfies quasi-injective descent iff
\begin{enumerate}
\item{$\mathbf{hom}_{\textnormal{\textbf{M}}^{\circ}|_{U}}(x, y)$ satisfies injective descent for each $x, y \in \mathrm{Ob}(\textnormal{\textbf{M}}^{\circ})(U)$.}
\item{The map $\mathrm{Ho}(M(U)) \rightarrow \mathrm{Ho}(\textnormal{\textbf{LAX}}_{\Delta} (M \circ \tilde{C}(\phi)))$ is essentially surjective for each single element cover generated by a morphism $\phi$.}
\item{For each $\alpha$-bounded product $\coprod_{\alpha} U_{\alpha}$ of elements of $\mathscr{C}$
$$
\mathrm{Ho}(M\left(\coprod_{\alpha} U_{\alpha}\right)) \rightarrow \mathrm{Ho}\left( \prod_{\alpha} M(U_{\alpha})\right) 
$$
is essentially surjective. 
}
  \end{enumerate}
  
\end{theorem}

\begin{proof}
We apply Theorem~\ref{thm4.12}. Conditions (2) and (3) above are equivalent to conditions (2) and (3) of Theorem~\ref{thm4.12} by Corollary~\ref{cor5.7}.
It follows from Theorem~\ref{thm1.9} 2. that for all $x, y \in \mathrm{Ob}(\textbf{M})(U)$, there is a zig-zag of sectionwise weak equivalences connecting 
$\textbf{hom}_{\textbf{M}}(x, y) $ and $\mathrm{Map}_{\mathfrak{B}(\textbf{M}^{\circ})}(x, y)$. Thus, condition (1) above is equivalent to condition (1) of Theorem~\ref{thm4.12}.  

\end{proof}

\begin{remark}\label{rmk5.9}
 Suppose that we have an exact $I$-indexed diagram $M$ of cofibration categories, such that for each $U \in \mathrm{Ob}(\mathscr{C})$, $M(U)$ admits a (two-sided) calculus of fractions. Then Theorem~\ref{thm5.3} and Corollary~\ref{cor5.6} hold for $M$. 
 
To see this, note that Theorem~\ref{thm5.4} holds for cofibration categories $N$ admitting a 2-sided calculus of fractions. Recall the functor $i_{0}$ from Theorem~\ref{thm5.5}. We have weak equivalences connecting $i_{0}\mathfrak{B}(L_{H}(N^{c}))$ and $L_{C}(N)$ by \cite[Remark 4.6]{Kapulkin-Szumilo}. The statement now follows from \cite[Corollary 6.4]{DK-Calculating} and the characterization of the image of $i_{0} \mathfrak{B}$ found in the paper \cite{Bergner-Arising} (in particular, see Theorem 7.3 and the discussion at the end of Section 4). Thus the proof of Theorem~\ref{thm5.3} in \cite{Model-Holims} goes through verbatim for $M$. So does the proof of Corollary~\ref{cor5.6}.  

We will see in the final section how this remark can allow us to avoid the use of Grothendieck universes, by replacing presheaf of model categories by an appropriate presheaf of small subcategories.  
\end{remark}

\section{An Example: the Higher Stack of Unbounded Complexes}\label{sec8}

In this section, we show that unbounded complexes on a reasonable ringed site form a higher stack (Corollary~\ref{cor6.12}). The homotopy theory of unbounded complexes is equivalent to the homotopy of simplicial $\mathcal{R}$-module spectra,  where $\mathcal{R}$ is the structure sheaf. Thus, the problem reduces to showing the existence of the stack of simplicial $\mathcal{R}$-module spectra. We can apply Theorem~\ref{thm5.8} because the model structures involved are simplicial. The main difficulty is verifying condition (2) of the theorem; the technical heart of the section is Lemma~\ref{lem6.5}, where we verify this condition for simplicial $\mathcal{R}$-modules, and hence non-negatively graded complexes. Condition (2) for $\mathcal{R}$-module spectra then easily follows. 

In the final part of the section, we will explain how one can deal with the set-theoretic issues inherent in talking about presheaves of model categories, by showing how we can replace them with small subcategories in a way that does not lose homotopy-theoretic information. 
\\

In this section, we will fix a Grothendieck site $\mathscr{C}$, which satisfies the hypotheses of Definition~\ref{def4.11}. We will fix a sheaf of commutative, unital rings $\mathcal{R}$ on $\mathscr{C}$. 
Given a sheaf of rings $\mathcal{T}$ on a site $\mathscr{E}$, we will denote categories of the sheaves of $\mathcal{T}$-modules and sheaves of simplicial $\mathcal{T}$-modules, respectively, by 
$$
\textbf{Sh}_{\mathcal{T}}, \, \, \, s\textbf{Sh}_{\mathcal{T}}.
$$

Given an abelian category $\mathcal{A}$, we will write $\mathrm{Ch}_{+}(\mathcal{A})$, $\mathrm{Ch}(\mathcal{A})$, respectively, for the categories of non-negatively graded and unbounded chain complexes. For instance, if $\mathcal{T}$ is a sheaf of rings on a site, we write $\mathrm{Ch}_{+}(\textbf{Sh}_{\mathcal{T}}), \mathrm{Ch}(\textbf{Sh}_{\mathcal{T}})$, respectively, for non-negatively graded and unbounded complexes of $\mathcal{T}$-modules. We will denote the derived categories of non-negatively graded and unbounded complexes, respectively, by $D_{+}(\mathcal{A}), D(\mathcal{A})$. We will write $\tau_{\ge n} : \mathrm{Ch}(\mathcal{A}) \rightarrow \mathrm{Ch}(\mathcal{A})$ for the intelligent truncation of a complex at level $n$, i.e. the truncation functor that preserves homology at level n. 
\\

Let $\phi : V \rightarrow U$ be morphism generating a covering in $\mathscr{C}$. We write $\Omega_{\phi} : \Delta^{*} \rightarrow \mathrm{Cat}$ for the functor  $$ [\textbf{n}] \mapsto \textbf{Sh}_{\mathcal{R}|_{\tilde{C}(\phi)_{n}}},$$
where $\tilde{C}(\phi)$ is as defined in Example~\ref{exam2.13}, and such that for each map $\theta : [\textbf{n}] \rightarrow [\textbf{m}]$, $\Omega_{\phi}(\theta)$ is the restriction functor. We can also define a functor $\Gamma_{\phi} : \Delta^{*} \rightarrow \mathrm{Cat}$ such that $[\textbf{n}] \mapsto \mathrm{Ch}_{+}(\Omega_{\phi}(\textbf{n}))$. In this case  $\mathrm{Ch}_{+}(\mathfrak{L}_{\Delta^{*}}(\Omega_{\phi}) )$ can be identified with $\mathfrak{L}_{\Delta^{*}}(\Gamma_{\phi})$ (as ordinary categories).

We write 
$
\mathfrak{L}_{\Delta^{*}}^{\mathrm{cart}}(\Omega_{\phi})$ for the full subcategory of $\mathfrak{L}_{\Delta^{*}}(\Omega_{\phi})$ (see Definition~\ref{def5.1}) whose  objects are those $(x_{\alpha}, u_{\alpha, \beta}^{\theta})$ such that each $u_{\alpha, \beta}^{\theta}$ is an isomorphism. We call these objects \textbf{cartesian objects}. 

\begin{lemma}\label{lem6.1}
The functor $\textnormal{\textbf{Sh}}_{\mathcal{R}|_{U}} \rightarrow \mathfrak{L}_{\Delta^{*}}^{\mathrm{cart}}(\Omega_{\phi})$ induced by restriction is an equivalence of categories. 
\end{lemma}
\begin{proof}
This follows from the glueing theorem for sheaves on a site (see \cite[Lemmas 7.26.1 and 7.26.4]{Stacks}).
 \end{proof} 
 
 \begin{lemma}\label{lem6.2}
 $\mathfrak{L}_{\Delta^{*}}(\Omega_{\phi})$ has enough injectives. 
 \end{lemma}
 
 \begin{proof}
 
 By \cite[Theorem 10.1, Proposition 1.9]{Tohoku}
 It suffices to show that the abelian category admits a system of generators and satisfies axiom AB5. The second follows since sheaves on a site are an AB5 category, and filtered colimits and finite limits in $\mathfrak{L}_{\Delta^{*}}\Omega_{\phi}$ are calculated sectionwise. For each $n \in \mathbb{N}$, let $A$ be an object of $\Omega_{\phi}([\textbf{n}])$ and define an element $\mathfrak{U}_{n, A} = (y_{[\textbf{k}]}, \theta^{\alpha}_{[\textbf{k}], [\textbf{m}]})$ by  
$$ y_{[\textbf{m}]} =
\left\{
	\begin{array}{ll}
		0  & \mbox{if } m < n \\
		(A) |_{\tilde{C}(\phi)_{m}} & \mbox{if } m \ge n
	\end{array}
\right.
$$
 and such that $\theta_{[\textbf{m}], [\textbf{k}] }^{\alpha}$ is the identity for $m, k \ge n$ and is the zero map otherwise. For each $n \in \mathbb{N}$ choose a generator $G_{n}$ of $\Omega_{\phi}([\textbf{n}])$.
 The $\mathfrak{U}_{n, G_{n}}$'s are a system of generators for $\mathfrak{L}_{\Delta^{*}}\Omega_{\phi}$.
  
    \end{proof}

 \begin{lemma}\label{lem6.3}
 Let $r^{*} : \mathrm{Ch}_{+}(\mathbf{Sh}_{\mathcal{R}|_{U}}) \rightarrow \mathrm{Ch}_{+}(\mathfrak{L}_{\Delta^{*}}(\Omega_{\phi}))$ be the map induced by restriction. The essential image of $r^{*}$ in the derived category $D_{+}(\mathfrak{L}_{\Delta^{*}}(\Omega_{\phi}))$ can be identified with $\textnormal{\textbf{LAX}}_{\Delta^{*}}(\Gamma_{\phi})$.
 \end{lemma}
  
 \begin{proof} (c.f. \cite[5bis 2.2.7]{SGA4}).
Recall that $\mathrm{Ch}_{+}(\mathfrak{L}_{\Delta^{*}}(\Omega_{\phi}) )$ can be identified with $\mathfrak{L}_{\Delta^{*}}(\Gamma_{\phi})$.  
 Note that restriction of sites $r^{*}$ has a right adjoint $r_{*}$. 
It suffices to show that for $E \in \textbf{LAX}_{\Delta^{*}}(\Gamma_{\phi}))$, the derived counit 
$$
r^{*}\textbf{R}(r_{*})(E) \rightarrow E
$$
is a quasi-isomorphism. 

We will first assume that $E$ has bounded homology, and proceed by induction on the length $l$ of the interval in which $H_{i}(E) \neq 0$. 
First consider the case $l = 1$. In this case, $H_{i}(E)[-i] \cong E$ in the derived category, and $H_{i}(E)$ is a cartesian object of $\mathfrak{L}_{R} (\Omega_{R})$. Thus, the result follows from Lemma~\ref{lem6.1}. 

In general, suppose that $n$ is the largest integer such that $H_{n}(E) \neq 0$. We have a distinguished triangle
$$
H_{n}(E)[-n] \rightarrow \tau_{\le n-1}(E) \rightarrow E  \rightarrow H_{n}(E)[1-n],
$$ 
where $\tau_{\le n}$ is the `intelligent' truncation at level n that preserves homology. The fact that the derived counit is an isomorphism on the first and third terms implies that it is an isomorphism on $E$. 
  
In general, suppose that we have a complex $E$. We want to show that the counit map induces an isomorphism
$ H_{n}(r^{*}\textbf{R}(r_{*})E) \rightarrow H_{n}(E)$ for each $n$.
  We have a diagram:
   $$
      \xymatrix
      {
   r^{*}\textbf{R}(r_{*})(\tau_{\le n}E)   \ar[d] \ar[r]_>>>>>>{\simeq} &  \tau_{\le n}E \ar[d] \\
    r^{*}\textbf{R}(r_{*})(E)  \ar[r] & E 
      }
   $$
    By the two out of three property, it suffices to show that 
    $$ H_{n-1}(r^{*}\textbf{R}(r_{*})(\tau_{\le n}(E))) \rightarrow H_{n-1}(r^{*}\textbf{R}(r_{*})(E))$$ 
      is an isomorphism for each $n$. The functor $r^{*}$ preserves and reflects quasi-isomorphisms. Thus, it suffices to show that $H_{n-1}(\textbf{R}(r_{*})(\tau_{\le n}E)) \rightarrow H_{n-1}(\textbf{R}(r_{*})(E) )$ is an isomorphism.  But we have a distinguished triangle 
 $$
\tau_{\le n}(E) \rightarrow E  \rightarrow (E / \tau_{\le n} E) \rightarrow \tau_{\le n}(E)[1],
$$
    where $(E / \tau_{\le n} E)$ is zero in degree $\le n$. $H_{n}(E/ \tau_{\le n}E) \cong 0$ for $i \le n$. It follows from the long exact sequence of homology that $ H_{n-1}(r^{*}\textbf{R}(r_{*})(\tau_{\le n}(E))) \rightarrow H_{n-1}(r^{*}\textbf{R}(r_{*})(E))$ is an isomorphism.

      \end{proof}

\begin{theorem}\label{thm6.4} \textnormal{(}see \textnormal{\cite[Lemma 1.2]{JardineChain}}\textnormal{)}.
There is a cofibrantly generated, proper, simplicial model structure on $s{\mathbf{Sh}}_{\mathcal{R}}$ in which the weak equivalences are the local weak equivalences and the fibrations are the injective fibrations.
The simplicial hom is the usual simplicial hom for the Jardine model structure. 

The Dold-Kan correspondence
$$
N :s\mathbf{Sh}_{\mathcal{R}} \rightleftarrows \mathrm{Ch}_{+}(\mathbf{Sh}_{\mathcal{R}}) : \Gamma
$$
\noindent
and the above model structure induce a model structure on $ \mathrm{Ch}_{+}(\mathbf{Sh}_{\mathcal{R}})$ in which the weak equivalences are the quasi-isomorphisms. 
\end{theorem}

Consider the diagram of model categories $s\mathrm{SH}_{\phi}$ defined by 
$$
[\textbf{n}] \mapsto s\textbf{Sh}_{\mathcal{R}|_{\tilde{C}(\phi)_{n}}} 
$$
and endowed with the model structure of Theorem~\ref{thm6.4}. This is a presheaf of left Quillen functors; restriction preserves local weak equivalences (\cite[Corollary 5.26]{local}) and monomorphisms. Thus, the right adjoint of restriction preserves injective fibrations and trivial injective fibrations. 

Recall from Section~\ref{sec7} the model structure on $\mathfrak{L}_{\Delta^{*}} (\mathrm{sSH}_{ \phi})$. Then it is clear from the definition that there is a Quillen adjunction  
$$
(r_{s\mathrm{SH}}^{\phi})^{*}  : (s\textbf{Sh}_{\mathcal{R}|_{U}}) \rightleftarrows \mathfrak{L}_{\Delta^{*}}(s\mathrm{SH}_{\phi}) : (r_{s\mathrm{SH}}^{\phi})_{*},
$$
where the left adjoint is induced by restriction. 

\begin{lemma}\label{lem6.5}
Suppose that $\sigma \in \mathbf{LAX}_{\Delta^{*}}(s\mathrm{SH}_{\phi})$.Then the derived counit map 
$$
\textbf{L}((r_{s\mathrm{SH}}^{\phi})^{*} )\textbf{R}((r_{s\mathrm{SH}}^{\phi})_{*})(\sigma) \rightarrow  \sigma
$$
is an isomorphism in  $\mathrm{Ho}(\mathfrak{L}_{\Delta^{*}}(s\mathrm{SH}_{\phi}))$.
\end{lemma}

\begin{proof}

Suppose that $m : A \rightarrow B$ is a trivial cofibration in $s\mathrm{SH}_{\phi}([\textbf{n}])$. Then this induces a trivial cofibration $m' : \mathfrak{U}_{n, A} \rightarrow \mathfrak{U}_{n, B}$, where $\mathfrak{U}_{n, (-)}$ is as in Lemma~\ref{lem6.2}. 

Every fibration $ (x_{[\textbf{n}]}, \theta^{\alpha}_{[\textbf{n}], [\textbf{m}]}) \rightarrow (y_{[\textbf{n}]}, \gamma^{\alpha}_{[\textbf{n}], [\textbf{m}]})$ has the right lifting property with respect to $m'$, and hence each  $x_{[\textbf{n}]} \rightarrow y_{[\textbf{n}]}$ is a fibration for the model structure on $\mathrm{sSH}_{\phi}([\textbf{n}])$.  

It follows from the last sentence of the preceding paragraph, and the right properness of the model structure of Theorem\ref{thm6.4}, that the model structure of $\mathfrak{L}_{\Delta^{*}}(\mathrm{sSH}_{\phi})$ is right proper. Its homotopy category thus admits a description in terms of a `cocycle' category (\cite[Theorem 6.5]{local}). Via the Dold-Kan correspondence, this description recovers the derived category $D_{+}(\mathfrak{L}_{\Delta^{*}}(\Omega_{\phi}))$. The result follows from Lemma~\ref{lem6.3}. 
\end{proof}

We write $\mathcal{R} : s\textbf{Pre}(\mathscr{C}) \rightarrow s\textbf{Sh}_{\mathcal{R}}$ for the free $\mathcal{R}$-module functor, which is left adjoint to the forgetful functor.
 Given two simplicial $\mathcal{R}$-modules $N, M$, we write $N \otimes M$ for their tensor product.  Given a simplicial presheaf $K$ and a simplicial $\mathcal{R}$-module $M$, we write $K \otimes M$ for $\mathcal{R}(K) \otimes M$. 

\begin{definition}\label{def6.6}
A \textbf{spectrum in chain complexes of $\mathcal{R}$-modules} $A$ consists of non-negatively graded chain complexes $A^{n}$ of $\mathcal{R}$-modules, for all $n \ge 0,$ together with homomorphisms $\sigma : A^{n}[-1] \rightarrow A^{n+1}$ called \textbf{bonding maps}. 
A map $f : A \rightarrow B$ of spectra in chain complexes consists of maps $f: A^{n} \rightarrow B^{n}$ such that 
$$
\xymatrix
{
A^{n}[-1] \ar[rr]_{f[-1]} \ar[d]_{\sigma} && B^{n}  [-1] \ar[d]_{\sigma} \\
A^{n+1} \ar[rr]_{f} && B^{n+1}
}
$$

A \textbf{simplicial $\mathcal{R}$-module spectrum} $A$ consists of simplicial $\mathcal{R}$-modules sheaves $A^{n}, n \ge 0$, together with simplicial $\mathcal{R}$-module homomorphisms $\sigma : S^{1} \otimes A^{n} \rightarrow A^{n+1}$, called \textbf{bonding maps}. A map of $f : A \rightarrow B$ of simplicial $\mathcal{R}$-modules consists of maps $f : A^{n} \rightarrow B^{n}$ such that
$$
\xymatrix
{
S^{1} \otimes A^{n}  \ar[r]_{f \otimes id} \ar[d]_{\sigma} & S^{1} \otimes B^{n} \ar[d]_{\sigma} \\
A^{n+1} \ar[r]_{f} & B^{n+1}
}
$$
commutes.
\end{definition}

 We write $ \textbf{Spt}(s\textbf{Sh}_{\mathcal{R}}))$ and $\textbf{Spt}(\mathrm{Ch}_{+}(\textbf{Sh}_{\mathcal{R}}))$ for the categories of simplicial $\mathcal{R}$-module spectra and of spectra in chain complexes of $\mathcal{R}$-modules, respectively. The initial object in $\textbf{Spt}(s\textbf{Sh}_{\mathcal{R}})$  is denoted $\mathbb{S}$ (it is called the sphere spectrum) and satisfies $(\mathbb{S})^{n} = \mathcal{R}(S^{1})^{\otimes n}$ with bonding maps the identity. 

Given a simplicial $\mathcal{R}$-module spectrum $X$ and a simplicial presheaf $K$, we write $K \otimes X$ for the simplicial $\mathcal{R}$-module spectrum that is the sheafification of
$$
( K \otimes X)^{n} =  \mathcal{R}(K) \otimes X^{n}
$$
and whose bonding maps come from those of $X$. Given a simplicial presheaf $K$, write $\Sigma^{\infty}K =  K \otimes \mathbb{S}$.

Given $X \in \textbf{Spt}(s\textbf{Sh}_{\mathcal{R}})$, there is a presheaf of stable homotopy groups $\pi_{n}^{s}(X)$, defined by 
$$
U \mapsto \pi_{n}^{s}(X(U))
$$
(see \cite[pg. 377]{local}).
We say that a map of simplical $\mathcal{R}$-module spectra $f: A \rightarrow B$ is a \textbf{local stable equivalence} iff 
$L^{2} \pi_{n}^{s}(f)$ is an isomorphism for each $n \ge 0$. We say $f$ is a \textbf{local strict equivalence} iff each $A^{n} \rightarrow B^{n}$ is a local weak equivalence.

\begin{theorem}\label{thm6.7}\textnormal{\cite[Lemma 3.1 and Theorem 3.6]{JardineChain}}.
There is a cofibrantly generated, proper simplicial model structure on  $\mathbf{Spt}(s\mathbf{Sh}_{\mathcal{R}})$ in which the weak equivalences are the local stable equivalences. The cofibrations are maps $f: A \rightarrow B$ such that:
\begin{enumerate}
\item{$A^{0} \rightarrow B^{0}$ is a cofibration in the model structure of \textnormal{Theorem~\ref{thm6.4}}.}
\item{$(S^{1} \otimes B^{n}) \cup_{(S^{1} \otimes A^{n})} A^{n+1} \rightarrow B^{n+1}$ is a cofibration for the model structure of \textnormal{Theorem~\ref{thm6.4}} for all $n \ge 0$.}
\end{enumerate} 
The simplicial hom is given by $\mathbf{hom}(A, B)_{n} = \mathrm{hom}(\Delta^{n} \otimes A, B)$. We call this the \textbf{stable model structure} on  $\mathbf{Spt}(s\mathbf{Sh}_{\mathcal{R}})$. 

There also exists a cofibrantly generated, proper simplicial model structure, called the \textbf{strict model structure}, in which the weak equivalences are the local strict equivalences and the cofibrations are the same as in the stable model structure. The simplicial hom is the same as well. 
\end{theorem}

The purpose of introducing the category of simplicial $\mathcal{R}$-module spectra is to provide a simplicial model category whose homotopy category is equivalent to the derived category of unbounded complexes. Following \cite[Sections 2 and 3]{JardineChain}, we introduce an appropriate model structure on $\textbf{Spt}(\mathrm{Ch}_{+}(\textbf{Sh}_{\mathcal{R}}))$, and describe how it is Quillen equivalent to the model structure of Theorem~\ref{thm6.7} and an appropriate model structure on $\mathrm{Ch}(\textbf{Sh}_{\mathcal{R}})$. 
\\

We have an adjoint pair 
$$
S : \textbf{Spt}(\mathrm{Ch}_{+}(\textbf{Sh}_{\mathcal{R}})) \rightleftarrows  \mathrm{Ch}(\textbf{Sh}_{\mathcal{R}}) : T
$$
where $T(A)^{n} = \tau_{\ge -n}(A)[-n]$ and the bonding map $\sigma : T(A)^{n}[-1] \rightarrow T(A)^{n+1}$ is the natural map $$(\tau_{\ge -n}(A)[-n])[-1] = \tau_{\ge -n}(A)[-n-1] \rightarrow \tau_{\ge -n-1}(A)[-n-1].$$ Moreover, $S(B)$ is the filtered colimit of the diagram:
$$
A^{0} \xrightarrow{\sigma[1]} A^{1}[1] \xrightarrow{\sigma[2]} A^{2}[2] \xrightarrow{\sigma[3]} \cdots A^{n}[n] \xrightarrow{\sigma[n+1]} \cdots. 
$$

We call a morphism $f$ of $\textbf{Spt}(\mathrm{Ch}_{+}(\textbf{Sh}_{\mathcal{R}}))$ a \textbf{local stable equivalence} iff $S(f)$ is a quasi-isomorphism.

\begin{theorem}\label{thm6.8} \textnormal{(}see \textnormal{\cite[Theorems 2.1 and 2.6]{JardineChain}}\textnormal{)}.
There is a right proper model structure on $\mathbf{Spt}(\mathrm{Ch}_{+}(\mathbf{Sh}_{\mathcal{R}}))$, in which the weak equivalences. The cofibrations are maps such that 
\begin{enumerate}
\item{$A^{0} \rightarrow B^{0}$ is a cofibration in the model structure of \textnormal{Theorem~\ref{thm6.4}}.}
\item{$ B^{n}[-1] \cup_{ A^{n}[-1]} A^{n+1} \rightarrow B^{n+1}$ is a cofibration for the model structure of \textnormal{Theorem~\ref{thm6.4}} for all $n \ge 0$.}
\end{enumerate}

Moreover, $\mathrm{Ch}(\mathbf{Sh}_{\mathcal{R}})$ has a model structure in which the weak equivalences are quasi-isomorphisms, and the fibrations are maps $f$ such that $T(f)$ is a fibration in the model structure on $\mathbf{Spt}(\mathrm{Ch}_{+}(\mathbf{Sh}_{\mathcal{R}}))$. The adjoint pair $S \dashv T$ forms a Quillen equivalence between these two model structures. 

\end{theorem}

There is also an adjoint pair 
$$
\textbf{N} : \mathbf{Spt}(s\textbf{Sh}_{\mathcal{R}})  \rightleftarrows \mathbf{Spt}(\mathrm{Ch}_{+}(\mathbf{Sh}_{\mathcal{R}})) : \boldsymbol{\Gamma}
$$
where $\boldsymbol{\Gamma}(A)^{n} = \Gamma(A^{n})$. The bonding maps are 
$$
S^{1} \otimes \Gamma(A^{n})  \xrightarrow{\mu} \bar{W}(A^{n}) \cong \Gamma(A^{n}[-1]) \xrightarrow {\Gamma(\sigma)} \Gamma(A^{n+1})
$$
with $\bar{W}$ the classifying complex functor defined on \cite[pg. 271]{GJ2} and $\mu$ the homotopy equivalence of \cite[equation (3.1)]{JardineChain}.
 
\begin{theorem}\label{thm6.9} \textnormal{(}see \textnormal{\cite[Theorem 3.6]{JardineChain}}\textnormal{)}.
The adjoint pair $\textbf{N} \dashv \boldsymbol{\Gamma}$ form a Quillen equivalence between the stable model structure on $\mathbf{Spt}(s\mathbf{Sh}_{\mathcal{R}}))$ and the model structure on $\mathbf{Spt}(\mathrm{Ch}_{+}(\mathbf{Sh}_{\mathcal{R}}))$ from \textnormal{Theorem~\ref{thm6.8}}.
\end{theorem}

As a consequence of the preceding two theorems 

$$
\mathrm{Ho}(\textbf{Spt}(s\textbf{Sh}_{\mathcal{R}}) )
$$
is equivalent to the usual derived category of $\mathcal{R}$-modules. 

Note that restriction of sites along a morphism $\psi : W \rightarrow S$,
 $$
\textbf{Spt}(s\textbf{Sh}_{\mathcal{R}|_{S}}) \rightarrow \textbf{Spt}(s\textbf{Sh}_{\mathcal{R}|_{W}}),
$$ is a left Quillen functor. It clearly preserves local stable equivalences, and it also preserves cofibrations since $\mathrm{sSH}_{\phi}$ is a left Quillen presheaf.  Thus, we have a $\Delta^{*}$-indexed left Quillen presheaf of model categories $\mathrm{SPT}_{\phi}$ defined by $$[n] \mapsto \textbf{Spt}(s\textbf{Sh}_{\mathcal{R}|\tilde{C}(\phi)_{n}} ) = \textbf{Spt}(s\textbf{Sh}_{\mathcal{R}|_{V \times_{U} \cdots \times_{U} V}})$$ with $$\mathrm{SPT}_{\phi}(\theta) : \mathrm{SPT}_{\phi}([\textbf{n}]) \rightarrow \mathrm{SPT}_{\phi}([\textbf{m}])$$ induced by restriction of sites.

\begin{lemma}\label{lem6.10}
Consider the Quillen adjunction 
$$
(r_{\mathrm{SPT}}^{\phi})^{*} :  \mathbf{Spt}( s\mathbf{Sh}_{\mathcal{R}}) \rightleftarrows \mathfrak{L}_{\Delta}(\mathrm{SPT}_{\phi}) : (r_{\mathrm{SPT}}^{\phi})_{*}
$$
where the left adjoint is induced from restriction. The essential image is $\mathbf{LAX}_{\Delta^{*}}(\mathrm{SPT}_{\phi})$.

\end{lemma}

\begin{proof}

Given an element $X = (x_{n}, \theta_{n, m}^{\alpha})$ of $\mathrm{SPT}_{\phi}$, we will write $X^{k}$ for the element  $((x_{n})^{k}, (\theta_{n, m}^{\alpha})^{k} )$ of $\mathrm{sSH}_{\phi}$.  

It suffices to show that for each fibrant object $X \in \textbf{LAX}_{\Delta^{*}}(\mathrm{SPT}_{\phi})$, and choice of cofibrant replacement $Y \rightarrow (r_{\mathrm{SPT}})_{*}(X)$, the map 
$$
(r_{\mathrm{SPT}}^{\phi})^{*}(Y) \rightarrow X 
$$
is a local stable equivalence.  Every strict local equivalence of simplicial $\mathcal{R}$-module spectra is a local stable equivalence (\cite[Lemma 2.2]{JardineChain}). Thus, we can take cofibrant replacement in the strict model structure of Theorem~\ref{thm6.7}. That is, each $Y^{n} \rightarrow (r_{\mathrm{SPT}}^{\phi})_{*}(X)^{n} $ is a local weak equivalence. 

Indeed, for $A \in s\textbf{Pre}_{\mathcal{R}|_{U}}$, we have  natural isomorphisms 
\begin{align*}
 \mathrm{hom}(A, (r_{s\mathrm{SH}}^{\phi})_{*}(X^{n}) ) \cong
\mathrm{hom}((\Sigma^{\infty}((r_{s\mathrm{SH}}^{\phi})^{*}(A)[-n]), X) \cong \\ \mathrm{hom}((r_{\mathrm{SPT}}^{\phi})^{*}(\Sigma^{\infty}(A[-n])), X)  \cong  \mathrm{hom}(\Sigma^{\infty}(A[-n]), (r_{\mathrm{SPT}}^{\phi})_{*}(X) ) \cong \\ \mathrm{hom}(A, (r_{\mathrm{SPT}}^{\phi})_{*}(X)^{n}).
\end{align*}

Thus, by Yoneda, $((r_{\mathrm{SPT}}^{\phi})_{*}(X))^{n} \cong (r_{s\mathrm{SH}}^{\phi})_{*}(X^{n})$. Note that $X^{n}$ is fibrant for the model structure on $\mathfrak{L}_{\Delta^{*}}(s\mathrm{SH}_{\phi})$, so that we conclude that the map $(r_{\mathrm{SPT}}^{\phi})^{*}(Y) \rightarrow X $ is a weak equivalence from \cite[Lemma 2.2]{JardineChain} and Lemma~\ref{lem6.5}.

\end{proof}

Let $\textbf{SPT}$ be the presheaf of simplicial model categories on $\mathscr{C}$, such that 
$$
U \mapsto  \textbf{Spt}(s\textbf{Sh}_{\mathcal{R}|_{U} }),
$$
equipped with the simplicial structure of Theorem~\ref{thm6.7}. 

\begin{theorem}\label{thm6.11}
$\mathfrak{B}(\mathbf{SPT}^{\circ})$ satisfies quasi-injective descent. 
\end{theorem}

\begin{proof}
We apply Theorem~\ref{thm5.8}. Condition (3) is trivial, and Condition (2) is Lemma~\ref{lem6.10}. We verify condition (1). It suffices to show that given $U \in \mathrm{Ob}(\mathscr{C}), x, y \in \textbf{SPT}^{\circ}(U)$, $\textbf{hom}_{\textbf{SPT}^{\circ}|_{U}}(x|_{U}, y|_{U})$ is injective fibrant. We want to solve lifting problems of the form 
$$
\xymatrix
{
A \ar[r] \ar[d] & \textbf{hom}_{\textbf{SPT}^{\circ}|_{U}}(x|_{U}, y|_{U}) \ar[d] \\
B \ar@{.>}[ur] \ar[r] & \, *
}
$$ 
where $A \rightarrow B$ is a local trivial cofibration. But this lifting problem is equivalent to a lifting problem of the form 
$$
\xymatrix
{
 A \otimes x|_{U} \ar[r] \ar[d] & y|_{U} \ar[d] \\
 B \otimes x|_{U}  \ar@{.>}[ur] \ar[r] & \,  *
}
$$ 
We can find  such a lift; the left and right vertical maps are, respectively, a trivial cofibration (by \cite[Lemmas 1.1 and 1.4]{JardineChain}) and a fibration (by definition) for the model structure of Theorem~\ref{thm6.7}. 
\end{proof}

Let $\mathrm{CH}$ be the left Quillen presheaf given by 
$$
U \mapsto \mathrm{Ch}(\textbf{Sh}_{\mathcal{R}|_{U}}),
$$
where $ \mathrm{Ch}(\textbf{Sh}_{\mathcal{R}|_{U}})$ is equipped with the model structure of Theorem~\ref{thm6.8}. 

\begin{corollary}\label{cor6.12}
$\mathfrak{B}(L_{H}(\mathrm{CH}^{c}))$ satisfies quasi-injective descent. 
\end{corollary}

\begin{proof}
There is a zig-zag of sectionwise Joyal equivalences connecting $\mathfrak{B}L_{H}(\textbf{SPT}_{0}^{c})$ and $\mathfrak{B}(\textbf{SPT}^{\circ})$ byTheorem~\ref{thm1.8} and \cite[Propositions 4.8 and 5.2]{DK-Function}. Furthermore, there is a natural transformation $\textbf{SPT}_{0} \rightarrow \mathrm{CH}$ which is a Quillen equivalence in each section by Theorem~\ref{thm6.8} and Theorem~\ref{thm6.9}. This induces a map $L_{H}(\textbf{SPT}_{0}^{c}) \rightarrow L_{H}(\textbf{CH}^{c})$ which is a sectionwise DK-equivalence (\cite[3.6, 8.4]{DK-Calculating}).

Recall that $L_{H}$ is obtained by applying the usual hammock localization, followed by fibrant replacement; all diagrams involved are presheaves of fibrant simplicial categories. Thus, applying the homotopy coherent nerve we get a collection of sectionwise Joyal equivalences connecting $\mathfrak{B}L_{H}(\mathrm{CH}^{c})$ and $\mathfrak{B}(\textbf{SPT}^{\circ})$. Thus, the result follows from Theorem~\ref{thm6.11}.

\end{proof}

The object $\mathfrak{B}(\textbf{SPT}^{\circ})$ is not a presheaf of small categories, and thus is problematic set-theoretically. In the next few paragraphs, we will demonstrate how to replace $\textbf{SPT}^{\circ}$ with a presheaf of small simplicial categories (without recourse to Grothendieck universes). The same arguments can be applied to other presheaves of model categories of interest. 

For a simplicial set $X$, the cardinality of $X$ is defined to be $$|X| = \underset{n \in \mathbb{N}}{\mathrm{sup}}(|X_{n}|).$$ For each simplicial presheaf $X$ and infinite cardinal $\alpha$, say that $X$ is $\alpha$-bounded if $$\underset{U \in \mathrm{Ob}(\mathscr{E})}{\mathrm{sup}}(|X(U)|) < \alpha.$$ Given a presheaf of $\mathcal{R}$-module spectra $Y$, we say $Y$ is $\alpha$-bounded iff each $Y^{n}$ is $\alpha$-bounded. We say that a cofibration $A \rightarrow B$ of simplicial sets is $\alpha$-bounded iff $B$ is. 

Given an uncountable, regular cardinal $\beta$, we write $\textbf{SPT}_{\beta}$ for the presheaf of simplicial categories such that $\textbf{SPT}_{\beta}(U)$ is the subsimplicial category, full in each simplicial degree, of $\textbf{SPT}(U)$ consisting of $\beta$-bounded objects. 

\begin{lemma}\label{lem6.13}
For sufficiently large regular cardinals $\beta$, the following hold:
\begin{enumerate}
 \item{ $(\mathbf{SPT}_{\beta})_{0}$  is an exact presheaf of cofibration categories.}
 \item{For each $U \in \mathrm{Ob}(\mathscr{C})$ $(\mathbf{SPT}_{\beta})_{0}(U)$ admits a two-sided calculus of fractions. }
   \end{enumerate}
\end{lemma}

\begin{proof}
For the first statement, it suffices to show that  $\textbf{SPT}_{\beta}(U)_{0}$ is a cofibration category for sufficiently large $\beta$. 
Let $\mathcal{K}$ be a family of generating cofibrations for the model structure on $\textbf{Spt}(s\textbf{Sh}_{\mathcal{R}}|_{U})$. Choose a cardinal $\lambda$, such that $\lambda > |Mor(\mathscr{C})|$ and each element of $\mathcal{K}$ is $\lambda$-bounded. Let $\beta = 2^{\lambda} +1$. It suffices to show 

\begin{enumerate}
\item{$\textbf{SPT}_{\beta}(U)_{0}$ is closed under pushout.}
\item{Functorial factorizations in $\textbf{Spt}(s\textbf{Sh}_{\mathcal{R}}|_{U})$, as a cofibration followed by a trivial fibration, preserve $\beta$-bounded objects.}
\end{enumerate}
The pushout of a diagram in $\textbf{Spt}(s\textbf{Sh}_{\mathcal{R}}|_{U})$ is obtained by applying sheafification to a presheaf theoretic pushout. Presheaf theoretic pushout clearly preserves $\beta$-bounded objects. Since $\beta > |Mor(\mathscr{C})|$ sheafification also preserves $\beta$-bounded objects. Thus the first statement holds. 

The second statement follows from arguments similar to those of \cite[Lemma 4.8 d.]{J1}. 

To show that $(\textbf{SPT}_{\beta})_{0}(U)$ admits a calculus of fractions, we apply \cite[Proposition 8.2]{DK-Calculating}. The classes of trivial cofibrations and trivial fibrations, respectively, provide the classes $W_{1}, W_{2}$ required by the proposition. The final condition is given by the existence of functorial factorizations in $(\textbf{SPT}_{\beta})_{0}(U)$.

\end{proof}

\begin{corollary}\label{cor6.14}
 For sufficiently large regular cardinals $\beta$, $\mathfrak{B}(\mathbf{SPT}_{\beta}^{\circ})$ satisfies quasi-injective descent.
\end{corollary}

\begin{proof}
Choose $\beta$ as in Lemma~\ref{lem6.13}. We will use Remark~\ref{rmk5.9} and Theorem~\ref{thm5.8}.  
For each $U \in \mathrm{Ob}(\mathscr{C}),$ $ x, y \in \textbf{SPT}_{\beta}^{\circ}(U)$, $\textbf{hom}_{\textbf{SPT}_{\beta}|_{U}}(x, y)$ satisfies injective descent by the proof of Theorem~\ref{thm6.11}. Moreover, for each  $\mathcal{U} \subseteq \mathrm{Ob}(\mathscr{C})$, we have an equivalence of categories 
$$
\textbf{SPT}_{\beta}(\coprod_{U \in \mathcal{U}} U) \simeq \prod_{U \in \mathcal{U}} \textbf{SPT}_{\beta}(U)
$$
provided that $\beta$ is larger than the cardinal $\alpha$ in Definition~\ref{def4.11}. 

Suppose that $\phi : V \rightarrow U$ generates a cover choose $X \in \textbf{LAX}_{\Delta^{*}}(\mathrm{SPT}_{\phi})$ such that it is $\beta$-bounded. 
  By the proof of Lemma~\ref{lem6.10}, we have a local stable equivalence
  $$
(r_{\mathrm{SPT}})^{*}Y \rightarrow X,
$$
where $Y \rightarrow (r_{\mathrm{SPT}})_{*}(X)$ is a cofibrant replacement. 
  To complete the proof, it suffices to show that each $(r_{\mathrm{SPT}})_{*}(X)^{n} \cong (r_{s\mathrm{SH}})_{*}(X^{n})$ is $\beta$-bounded (cofibrant replacement preserves $\beta$-bounded objects). Let $X^{n} =(x_{[\textbf{m}]}, \gamma_{[\textbf{k}], [\textbf{m}]}^{\theta})$. Then $r_{s\mathrm{SH}}(X^{n})$  is a limit
  $$
\mathrm{lim}_{m \in \mathbb{N}} s_{m}(x_{[\textbf{m}]}),
$$
 where $s_{m} :  s\textbf{Sh}_{\mathcal{R}|_{U \times_{V} U \cdots \times_{V} U}} \rightarrow s\textbf{Sh}_{\mathcal{R}|_{U}}  $   is the right adjoint of restriction of sites ($V$ appears $m$ times in the product). This functor is just the functor described in \cite[Proposition 11.3.1]{SGA4} (c.f. also \cite[Corollary 5.26]{local}), which preserves $\beta$-bounded objects. 

\end{proof} 

\section*{Acknowledgement}

The author is deeply indebted to the anonymous referee, whose comments led to a significant improvement in the paper.

%
%

\bibliographystyle{spmpsci}      
\bibliography{DESCENTBIB2.bib}   

\end{document}